\documentclass[reqno]{amsart}
\usepackage{amssymb, amsfonts}
\usepackage{enumerate}
\usepackage[usenames, dvipsnames]{color}
\usepackage[hypertex]{hyperref}

\numberwithin{equation}{section}

\newtheorem{theorem}{Theorem}[section]

\newtheorem{lemma}[theorem]{Lemma}

\theoremstyle{definition}
\newtheorem{remark}[theorem]{Remark}

\theoremstyle{definition}

\theoremstyle{definition}
\newtheorem{assumption}[theorem]{Assumption}

\makeatletter
\def\dashint{\operatorname%
{\,\,\text{\bf--}\kern-.98em\DOTSI\intop\ilimits@\!\!}}
\makeatother

\def\bC{\mathbb{C}}

\def\bR{\mathbb{R}}
\def\bZ{\mathbb{Z}}

\def\cD{\mathcal{D}}

\def\cL{\mathcal{L}}
\def\cM{\mathcal{M}}

\def\cU{\mathcal{U}}

\begin{document}

\title[Stokes system]{$L_q$-estimates for stationary Stokes system with coefficients measurable in one direction}

\author[H. Dong]{Hongjie Dong}
\address[H. Dong]{Division of Applied Mathematics, Brown University, 182 George Street, Providence, RI 02912, USA}

\email{Hongjie\_Dong@brown.edu}

\thanks{H. Dong was partially supported by the NSF under agreements DMS-1056737 and DMS-1600593.}

\author[D. Kim]{Doyoon Kim}
\address[D. Kim]{Department of Mathematics, Korea University, 145 Anam-ro, Seongbuk-gu, Seoul, 02841, Republic of Korea}

\email{doyoon\_kim@korea.ac.kr}

\thanks{D. Kim was supported by Basic Science Research Program through the National Research Foundation of Korea (NRF) funded by the Ministry of Education (2016R1D1A1B03934369).}



\begin{abstract}
We study the stationary Stokes system with variable coefficients in the whole space, a half space, and on bounded Lipschitz domains. In the whole and half spaces, we obtain a priori $\dot W^1_q$-estimates for any $q\in [2,\infty)$ when the coefficients are merely measurable functions in one fixed direction. For the system on bounded Lipschitz domains with a small Lipschitz constant, we obtain a $W^1_q$-estimate and prove the solvability for any $q\in (1,\infty)$ when the coefficients are merely measurable functions in one direction and have locally small mean oscillations in the orthogonal directions in each small ball, where the direction is allowed to depend on the ball.
\end{abstract}

\maketitle

\section{Introduction}

In this paper, we study $W^1_q$-estimates and the solvability of the stationary Stokes system with variable coefficients in the whole space, a half space, and on bounded Lipschitz domains. The regularity theory for the linear Stokes system has important applications in mathematical fluid dynamics, for instance, the Navier-Stokes equations. This theory has been extensively studied over the last fifty years by many authors. For the classical Stokes system with the Laplace operator in smooth domains, i.e.,
\begin{equation}
                        \label{eq9.59}
\begin{cases}
\Delta u+\nabla p=f\quad &\text{in}\,\,\Omega\\
\operatorname{div} u = g\quad &\text{in}\,\,\Omega
\end{cases}
\end{equation}
with the non-homogeneous Dirichlet boundary condition $u=\varphi$ on $\partial\Omega$,
we refer the reader to Lady{\v{z}}enskaya \cite{La59}, Sobolevski{\u\i} \cite{Sob60}, Cattabriga \cite{Ca61}, Vorovi{\v{c}} and Judovi{\v{c}} \cite{VJ61}, and Amrouche and Girault \cite{AG91}. In particular, for any $q\in (1,\infty)$, the following $W^1_q$-estimate was obtained by Cattabriga \cite{Ca61} for the system in a bounded $C^{2}$ domain $\Omega\subset\bR^3$:
$$
\|Du\|_{L_q(\Omega)}+\|p\|_{L_q(\Omega)}\le N\|f\|_{W^{-1}_q(\Omega)}+N\|g\|_{L_q(\Omega)}+
N\|\varphi\|_{W^{1-1/q}_q(\Omega)}.
$$
The proof is based on the explicit representation of solutions using fundamental solutions.
By using a result by Agmon, Douglis, and Nirenberg \cite{ADN64} for elliptic systems together with an interpolation argument, Cattabriga's result was later extended by Amrouche and Girault \cite{AG91} to a bounded $C^{1,1}$ domain $\Omega\subset \bR^d$, for any $d\ge 2$. The system \eqref{eq9.59} on a bounded Lipschitz domain was first studied by Galdi, Simader, and Sohr \cite{GSS94}. They proved $W^1_q$-estimates and solvability under the assumption that the Lipschitz constant of the domain is sufficiently small. The problem was studied by Fabes, Kenig, and Verchota \cite{FKV88} in the case of arbitrary Lipschitz domains with the range of $q$ restricted, using the layer potential method and Rellich identities. For this line of research, see also \cite{Shen95, BS95, MM01, MW12, GK15} and references therein, some of which obtain estimates in Besov spaces.

We are interested in the Stokes system with variable coefficients:
\begin{equation}
                                        \label{eq11.08}
\begin{cases}
\cL u + \nabla p = f + D_\alpha f_\alpha
\quad
&\text{in}\,\,\Omega,
\\
\operatorname{div} u = g
\quad
&\text{in}\,\,\Omega,
\end{cases}
\end{equation}
where $\Omega  \subseteq \bR^d$ and $\cL$ is a strongly elliptic operator, given by
$$
\cL u = D_\alpha \left( A^{\alpha\beta} D_\beta u \right),
\quad A^{\alpha\beta} = [A^{\alpha\beta}_{ij}]_{i,j=1}^d
$$
for $\alpha,\beta =1,\ldots,d$. Here and throughout the paper, we use the Einstein summation convention on repeated indices. Such type of systems were considered by Giaquinta and Modica \cite{MR0641818}, where they obtained various regularity results for both linear and nonlinear Stokes systems when the coefficients are sufficiently regular. Besides its mathematical interests, the system \eqref{eq11.08} is also partly motivated by the study of inhomogeneous fluids with density dependent viscosity (see, for instance, \cite{LS75, Li96, AGZ11}), as well as equations describing flows of shear thinning and shear thickening fluids with viscosity depending on pressure (see, for instance, \cite{FMR05, BMR07}).
See also Dindo{\v{s}} and Mitrea \cite{DM04} for its relation to the Navier-Stokes system in general Riemannian manifolds.

In this paper, we allow coefficients $A^{\alpha\beta}$ to be merely measurable in one direction. In particular, they may have jump discontinuities, so that the system can be used to model, for example, the motion of two fluids with interfacial boundaries. The system \eqref{eq11.08} is considered in the whole space, a half space, and on bounded Lipschitz domains. In the whole and half spaces, we obtain a priori $\dot W^1_q$-estimates for any $q\in [2,\infty)$ in the case that the coefficients are merely measurable functions in one fixed direction (see Theorem \ref{thm0307_1}). For the system on bounded Lipschitz domains with a small Lipschitz constant, we prove a $W^1_q$-estimate and solvability for \eqref{eq11.08} with any $q\in (1,\infty)$, when the coefficients are merely measurable functions in one direction and have locally small mean oscillations in the orthogonal directions in each small ball, with the direction depending on the ball (see Theorem \ref{thm0307_2}). These results extend the aforementioned results from \cite{GSS94} for the classical Stokes system \eqref{eq9.59}, and the recent study \cite{CL15} for \eqref{eq11.08} with coefficients having small mean oscillations in all directions.
We note that the class of coefficients considered in this paper was first introduced by Kim and Krylov \cite{MR2338417} and Krylov \cite{Kr09}, where they established the $W^2_p$-estimate for non-divergence form second-order elliptic equations in the whole space. Subsequently, such coefficients were also treated in \cite{MR2800569,MR2835999} for second- and higher-order elliptic and parabolic systems in regular and irregular domains.

Let us now provide an outline of the proofs of Theorems \ref{thm0307_1} and \ref{thm0307_2}. Our argument is completely different from the methods in \cite{Ca61,AG91,GSS94}, and is based on pointwise sharp and maximal function estimates in the spirit of \cite{Kr05, MR2338417,Kr09} for second-order elliptic equations. Such estimates rely on the $C^{1,\alpha}$ regularity of solutions to the homogeneous system. Here, the main difficulty is that because the coefficients are measurable in one direction, say $x_1$, it is impossible to obtain a H\"older estimate of the full gradient $Du$. To this end, instead of considering $Du$ itself, we exploit an idea given in \cite{MR2800569} to estimate certain linear combinations of $Du$ and $p$:
$$
D_{x'}u\quad \text{and}\quad A^{1\beta}D_\beta u + (p,0,\ldots,0)^{\operatorname{tr}}.
$$
Here and throughout the paper, $D_{x'}u$ denotes the partial derivative of $u$ in the $x_i$ direction, $i=2,\ldots,d$, where $x = (x_1, x') = (x_1, x_2, \ldots,x_d) \in \bR^d$.
For the Stokes system, the presence of the pressure term $p$ gives an added difficulty, because in the usual $L_2$-estimate, instead of $p$ one can only bound $p-(p)$ by $Du$, instead of $p$. See, for instance, Lemma \ref{lem0225_2}. Nevertheless, in Lemma \ref{lem3.6} we show that for the homogeneous Stokes system and any integer $k\ge 1$, the $L_2$-norm of $D_{x'}^k p$ in a smaller ball can be controlled by that of $Du$ in a larger ball. Finally, in order to deal with the system \eqref{eq11.08} in a Lipschitz domain, we apply a version of the Fefferman-Stein sharp function theorem for spaces of homogeneous type, which was recently proved in \cite{DK15} (cf. Lemma \ref{lem7.3}). Furthermore, we employ a delicate cut-off argument, together with Hardy's inequality, which was first used in \cite{MR2835999} and also in the recent paper \cite{CL15}.

In a subsequent paper, we will study {\em weighted} $W^1_q$-estimates and the solvability of the Stokes system \eqref{eq11.08} in more general Reifenberg flat domains, with the same class of coefficients. We note that an {\em a priori} $W^1_q$-estimate in Reifenberg flat domains was obtained in \cite{CL15}, under the condition that the coefficients have sufficiently small mean oscillations in all directions.

The remainder of this paper is organized as follows. We state the main theorems in the following section. Section \ref{sec_aux} contains some auxiliary results, including $L_2$-estimates and Caccioppoli type inequalities. In Section \ref{sec_04}, we prove interior and boundary $L_\infty$ and H\"older estimates for derivatives of solutions, while in Section \ref{sec_05} we establish the interior and boundary mean oscillation estimates for the system in the whole space and in a half space. Section \ref{sec_06} is devoted to the proof of Theorem \ref{thm0307_1}.
Finally, we consider the system \eqref{eq11.08} in a Lipschitz domain with a small Lipschitz constant in Section \ref{sec_07}.

We conclude this section by introducing some notation. We fix a half space to be $\bR^d_+$, defined by
$$
\bR^d_+ = \{ x = (x_1, x') \in \bR^d: x_1 > 0,\, x' \in \bR^{d-1} \}.
$$
Let $B_r(x_0)$ be the Euclidean ball of radius $r$ in $\bR^d$ centered at $x_0\in \bR^d$, and let $B^+_r(x_0)$ be the half ball
$$
B_r^+(x_0) = B_r(x_0) \cap \bR^d_+.
$$
A ball in $\bR^{d-1}$ is denoted by
$$
B'_r(x') = \{ y' \in \bR^{d-1}: |x' - y'| < r\}.
$$
We use the abbreviations
$B_r := B_r(0), B_r^+ := B_r^+(0)$ where $0\in \bR^d$, and $B_r' := B_r'(0)$ where $0 \in \bR^{d-1}$.

For a locally integrable function $f$, we define its average on $\Omega$ by
$$
(f)_\Omega = \frac{1}{|\Omega|} \int_{\Omega} f \, dx=\dashint_{\Omega} f \, dx.
$$

We shall use the following function spaces:
$$
W_q^1(\Omega):=\{f\in L_q(\Omega)\,:\,Df\in L_q(\Omega)\},
\quad
W_q^1(\Omega)^d=\big(W_q^1(\Omega)\big)^d.
$$
Finally, let $\mathring{W}_q^1(\Omega)$ be the completion of $C^\infty_0(\Omega)$ in $W_q^1(\Omega)$, and $\mathring{W}_q^1(\Omega)^d=\big(\mathring W_q^1(\Omega)\big)^d$.

\section{Main results}

In this section, we state our main results and the assumptions required for them.
Throughout this paper, the coefficients $A^{\alpha\beta}$ are bounded and satisfy the strong ellipticity condition, i.e., there exists a constant $\delta \in (0,1)$ such that
$$
|A^{\alpha\beta}| \le \delta^{-1},
\quad
\sum_{\alpha, \beta=1}^d \xi_\alpha \cdot A^{\alpha\beta} \xi_\beta \ge \delta \sum_{\alpha=1}^d |\xi_\alpha|^2
$$
for any $\xi_\alpha \in \bR^d$, $\alpha = 1, \ldots, d$.
Owing to the trace-extension theorem on Lipschitz domains, in the sequel we only consider the homogeneous boundary condition $u=0$ on $\partial \Omega$, without loss of generality.

We say that $(u,p)\in W_q^1(\Omega)^d\times L_q(\Omega)$ is a solution to \eqref{eq11.08} if for any $\psi = (\psi_1, \ldots, \psi_d) \in C_0^\infty(\Omega)^{d}$,
we have that
$$
- \int_\Omega D_\alpha \psi \cdot A^{\alpha\beta} D_\beta u \, dx - \int_\Omega p \, \operatorname{div} \psi \, dx = \int_\Omega f \cdot \psi \, dx - \int_\Omega f_\alpha \cdot D_\alpha \psi \, dx,
$$
where
$$
\operatorname{div} \psi = D_1 \psi_1 + \ldots + D_d \psi_d
\quad
\text{and}
\quad
D_\alpha \psi = (D_\alpha \psi_1, D_\alpha \psi_2, \ldots, D_\alpha \psi_d).
$$

Our first results concern a priori $L_q$-estimates of the Stokes system defined in $\bR^d$ or $\bR^d_+$, when the coefficients $A^{\alpha\beta}$ are merely measurable functions of only $x_1$.
In this case, throughout the paper, we set
\begin{equation}
							\label{eq0328_03}
\cL_0 u = D_\alpha \left( A^{\alpha\beta}(x_1) D_\beta u \right),
\end{equation}
where $A^{\alpha\beta}(x_1) = [A^{\alpha\beta}_{ij}(x_1)]_{i,j=1}^d$.
Note that we do not impose any regularity assumptions on $A^{\alpha\beta}(x_1)$.

\begin{theorem}
							\label{thm0307_1}
Let $q \in [2,\infty)$, and let $\Omega$ be either $\bR^d$ or $\bR^d_+$ and $A^{\alpha\beta} = A^{\alpha\beta}(x_1)$, i.e., $\cL = \cL_0$.
If $(u,p) \in W_q^1(\Omega)^d \times L_q(\Omega)$ satisfies
\begin{equation*}
\begin{cases}
\cL_0 u + \nabla p = D_\alpha f_\alpha
\quad
&\text{in}\,\,\Omega,
\\
\operatorname{div} u = g
\quad
&\text{in}\,\,\Omega,
\\
u = 0
\quad
&\text{on} \,\, \partial \Omega \quad \text{in case} \,\, \Omega = \bR^d_+,
\end{cases}
\end{equation*}
where $f_\alpha, g \in L_q(\Omega)$,
then we have that
\begin{equation}
							\label{eq0307_03}
\|Du\|_{L_q(\Omega)} + \|p\|_{L_q(\Omega)} \le N \left( \|f_\alpha\|_{L_q(\Omega)} + \|g\|_{L_q(\Omega)} \right),
\end{equation}
where $N=N(d,\delta,q)$.
\end{theorem}

\begin{remark}
In Theorem \ref{thm0307_1} we only consider the case that $q \in [2,\infty)$ to simplify the exposition and to present our approach in the most transparent way.
Indeed, if $q = 2$, the theorem holds even with measurable $A^{\alpha\beta}(x)$.
See Theorem \ref{thm0427}.
Thus, in the proof of Theorem \ref{thm0307_1} we focus on the case $q \in (2,\infty)$, the proof of which well illustrates, in the simplest setting, our arguments based on mean oscillation estimates together with the sharp function and the maximal function theorems.
One can prove the other case, with $q \in (1,2)$, by using Theorem \ref{thm0307_2} below. This is discussed in a more general setting with weights in \cite{DK18}.
\end{remark}

Next, when the Stokes system is defined in a bounded Lipschitz $\Omega$ with a small Lipschitz constant, we show that the system is uniquely solvable in $L_q(\Omega)$ spaces.
In this case, we allow coefficients not only to be measurable locally in one direction (near the boundary the direction is almost perpendicular to the boundary of the domain), but also to have small mean oscillations in the other directions.
To present this result, we require the following assumptions.

\begin{assumption}
							\label{assum0620}
For any $x_0 \in \partial \Omega$ and $0 < r \le R_0$,
there is a coordinate system depending on $x_0$ and $r$ such that in the new coordinate system we have
$$
\Omega\cap B_r(x_0) = \{x \in B_r(x_0)\, :\, x_1 >\phi(x')\},
$$
where $\phi: \bR^{d-1} \to \bR$ is a Lipschitz function with
$$
\sup_{\substack{x', \, y' \in B_r'(x_0')\\x' \neq y'}}\frac {|\phi(y')-\phi(x')|}{|y'-x'|}\le \frac{1}{16}.
$$
\end{assumption}

\begin{assumption}[$\gamma, \rho$]
							\label{assum0307_1}
Let $\gamma, \rho \in (0,1/16)$.
There exists $R_1 \in (0, R_0]$ satisfying the following.

\begin{enumerate}

\item For $x_0 \in \Omega$ and $0 < r \le \min\{R_1, \operatorname{dist}(x_0, \partial\Omega)\}$, there is a coordinate system depending on $x_0$ and $r$ such that in this new coordinate system we have that
\begin{equation}
							\label{eq0307_01}
\dashint_{B_r(x_0)} \Big| A^{\alpha\beta}(y_1,y') - \dashint_{B_r'(x_0')} A^{\alpha\beta}(y_1,z') \, dz' \Big| \, dx \le \gamma.
\end{equation}

\item For any $x_0 \in \partial \Omega$ and $0 < r \le R_1$,
there is a coordinate system depending on $x_0$ and $r$ such that in the new coordinate system we have that \eqref{eq0307_01} holds, and
$$
\Omega\cap B_r(x_0) = \{x \in B_r(x_0)\, :\, x_1 >\phi(x')\},
$$
where $\phi: \bR^{d-1} \to \bR$ is a Lipschitz function with
$$
\sup_{\substack{x', \, y' \in B_r'(x_0')\\x' \neq y'}}\frac {|\phi(y')-\phi(x')|}{|y'-x'|}\le \rho.
$$
\end{enumerate}
\end{assumption}

\begin{remark}
Clearly, Assumption \ref{assum0307_1} (2) is stronger than Assumption \ref{assum0620}.
However, we state these two assumptions separately for the following reason.
As seen in Theorem \ref{thm0307_2}, we specify the class of bounded Lipschitz domains for which the results of the theorem hold in terms of the flatness parameter $\rho$.
Thus, having two separate assumptions means that in Theorem \ref{thm0307_2} we specify a subclass of the class of domains satisfying Assumption \ref{assum0620}.
The necessity of such a hierarchy of classes of domains is that when determining the size of $\rho$ in Theorem \ref{thm0307_2}, we need some information about domains and their boundaries.
In particular, the maximal function and sharp function theorems on bounded domains we use in this paper require such information.
Thus, without Assumption \ref{assum0620}, the size of $\rho$ is to be determined by a set of parameters including $R_1$.
In this case, i.e., when $\rho$ is given by $R_1$, even a smooth domain $\Omega$ may not satisfy Assumption \ref{assum0307_1} (2) if $R_1$ is too large for the boundary to have $\rho$ flatness on $\Omega \cap B_{R_1}(x_0)$.
With two assumptions as above, every smooth domain satisfies Assumption \ref{assum0307_1} (2) for any $\rho$ by choosing a sufficiently small $R_1$.
\end{remark}

\begin{theorem}
							\label{thm0307_2}
Let $q ,q_1\in (1,\infty)$ satisfying $q_1 \ge  q d/(q+d)$, $K > 0$, and let $\Omega$ be bounded ($\operatorname{diam}\Omega \le K$).
Then, there exist constants $(\gamma, \rho) = (\gamma, \rho)(d,\delta,R_0, K, q) \in (0,1/16)$ such that, under Assumptions \ref{assum0620} and \ref{assum0307_1} $(\gamma, \rho)$, for $(u,p) \in  W_q^1(\Omega)^d \times L_q(\Omega)$
satisfying $(p)_\Omega = 0$ and
\begin{equation}
							\label{eq0307_04}
\begin{cases}
\cL u + \nabla p = f + D_\alpha f_\alpha
\quad
&\text{in}\,\,\Omega,
\\
\operatorname{div} u = g
\quad
&\text{in}\,\,\Omega,
\\
u = 0
\quad
&\text{on} \,\, \partial \Omega,
\end{cases}
\end{equation}
where $f \in L_{q_1}(\Omega)$, $f_\alpha, g \in L_q(\Omega)$, we have that
\begin{equation}
							\label{eq0328_02}
\|Du\|_{L_q(\Omega)} + \|p\|_{L_q(\Omega)} \le N \left( \|f\|_{L_{q_1}(\Omega)} + \|f_\alpha\|_{L_q(\Omega)} + \|g\|_{L_q(\Omega)} \right),
\end{equation}
where $N>0$ is a constant depending only on $d$, $\delta$, $R_0$, $R_1$, $K$, $q$, and $q_1$.
Moreover, for $f \in L_{q_1}(\Omega)$, $f_\alpha, g \in L_q(\Omega)$ with $(g)_\Omega = 0$, there exists a unique $(u,p) \in W_q^1(\Omega)^d \times L_q(\Omega)$ satisfying $(p)_\Omega = 0$ and \eqref{eq0307_04}.
\end{theorem}

\section{Auxiliary results}
                            \label{sec_aux}
In this section, we assume that the coefficients $A^{\alpha\beta}$ are measurable functions of $x \in \bR^d$. That is, no regularity assumptions are imposed on $A^{\alpha\beta}$.

We impose the following assumption on a bounded domain $\Omega\subset \bR^d$ in Lemma \ref{lem0225_1} below.

\begin{assumption}
							\label{assum0224_1}
For any $g \in L_2(\Omega)$ such that $\int_{\Omega} g \, dx = 0$, there exist $B g \in \mathring{W}_2^1(\Omega)^d$ and a constant $K_1>0$ depending only on $d$ and $\Omega$ such that
$$
\operatorname{div} Bg = g
\quad
\text{in}
\,\,
\Omega,
\quad
\|D(Bg)\|_{L_2(\Omega)} \le K_1 \|g\|_{L_2(\Omega)}.
$$
\end{assumption}

\begin{remark}
							\label{rem0229_1}
If $\Omega = B_R$ or $\Omega = B_R^+$, it follows from a scaling argument that the constant $K_1$ depends only on the dimension $d$.
If $\Omega$ is a bounded Lipschitz domain satisfying Assumption \ref{assum0620}, then Assumption \ref{assum0224_1} is satisfied with $K_1$ depending only on $d$, $R_0$, and $\operatorname{diam} \Omega$. If $1/16$ in Assumption \ref{assum0620} is replaced by $\rho$ and $\rho \in [0,\rho_0]$, the constant $K_1$ can be chosen so that it depends only on $d$, $R_0$, $\operatorname{diam}\Omega$, and $\rho_0$.
See, for instance, \cite{MR2101215}.
\end{remark}

\begin{lemma}
							\label{lem0225_1}
Let $\Omega \subset \bR^d$ be a bounded domain which satisfies Assumption \ref{assum0224_1}, and $f, f_\alpha, g \in L_2(\Omega)$ with $(g)_\Omega = 0$.
Then, there exists a unique $(u,p) \in W_2^{1}(\Omega)^d \times L_2(\Omega)$ with $(p)_\Omega = 0$ satisfying
$$
\begin{cases}
\cL u + \nabla p = f + D_\alpha f_\alpha
\quad
&\text{in}\,\,\Omega,
\\
\operatorname{div} u = g
\quad
&\text{in}\,\,\Omega,
\\
u = 0\quad &\text{on}\,\,\partial \Omega.
\end{cases}
$$
Moreover, we have that
$$
\|Du\|_{L_2(\Omega)} + \|p\|_{L_2(\Omega)} \le N \left( \|f\|_{L_2(\Omega)} + \|f_\alpha\|_{L_2(\Omega)} + \|g\|_{L_2(\Omega)} \right),
$$
where $N = N(d,\delta, K_1)$.
If
$\Omega = B_R(x_0)$, $x_0 \in \bR^d$, or $\Omega = B^+_R(x_0)$, $x_0 \in \partial \bR^d_+$,
then we have that
$$
\|Du\|_{L_2(\Omega)} + \|p\|_{L_2(\Omega)} \le N \left( R\|f\|_{L_2(\Omega)} + \|f_\alpha\|_{L_2(\Omega)} + \|g\|_{L_2(\Omega)} \right),
$$
where $N=N(d,\delta)$.
\end{lemma}

\begin{proof}
See, for instance, \cite[Lemma 3.1]{CL15}.
\end{proof}

As far as a priori estimates are concerned, one can have $\Omega=\bR^d$ or $\Omega = \bR^d_+$ in Lemma \ref{lem0225_1} if $f \equiv 0$.
In this case we do not necessarily need that the integral of $p$ over $\Omega$ is zero.
For completeness and later reference, we  state and prove this result in the theorem below.

Recall that we say that $(u,p) \in W_2^1(\Omega)^d \times L_2(\Omega)$ satisfies, for instance,
$\cL u + \nabla p = 0$ in $\Omega$, if
$$
\int_{\Omega} D_\alpha \psi \cdot A^{\alpha\beta} D_\beta u \, dx + \int_\Omega p\, \operatorname{div} \psi \, dx = 0
$$
for all $\psi \in C_0^\infty(\Omega)^d$.
One can easily see that $C_0^\infty(\Omega)^d$ can be replaced by $\mathring{W}_2^1(\Omega)^d$.

\begin{theorem}
							\label{thm0427}
Let $\Omega$ be either $\bR^d$ or $\bR^d_+$.
If $(u,p) \in W_2^1(\Omega)^d \times L_2(\Omega)$ satisfies
\begin{equation*}
\begin{cases}
\cL u + \nabla p = D_\alpha f_\alpha
\quad
&\text{in}\,\,\Omega,
\\
\operatorname{div} u = g
\quad
&\text{in}\,\,\Omega,
\\
u = 0
\quad
&\text{on} \,\, \partial \Omega \quad \text{in case} \,\, \Omega = \bR^d_+,
\end{cases}
\end{equation*}
where $f_\alpha, g \in L_2(\Omega)$,
then we have that
\begin{equation*}
\|Du\|_{L_2(\Omega)} + \|p\|_{L_2(\Omega)} \le N \left( \|f_\alpha\|_{L_2(\Omega)} + \|g\|_{L_2(\Omega)} \right),
\end{equation*}
where $N=N(d,\delta)$.
\end{theorem}

\begin{proof}
Since $u \in \mathring{W}_2^1(\Omega)^d$, we use $u$ as a test function to obtain that
$$
\int_\Omega D_\alpha u \cdot A^{\alpha\beta} D_\beta u \, dx + \int_\Omega p \, g \, dx = \int_\Omega f_\alpha \cdot D_\alpha u \, dx.
$$
From this, the ellipticity condition, and Young's inequality, we have that
\begin{equation}
							\label{eq0427_02}
\| Du\|_{L_2(\Omega)}^2 \le \varepsilon_0 \|p\|^2_{L_2(\Omega)} + \frac{N}{\varepsilon_0} \|g\|^2_{L_2(\Omega)} + N\|f_\alpha\|_{L_2(\Omega)}^2
\end{equation}
for any $\varepsilon_0 > 0$, where $N=N(d,\delta)$.

Now, for any $\varepsilon>0$, we can find $R> 0$ and $p_\varepsilon \in L_2(\Omega)$ such that $\operatorname{supp} p_\varepsilon \subset B_R$,
$\int_{\Omega} p_{\varepsilon} \, dx = 0$, and
$\|p - p_\varepsilon\|_{L_2(\Omega)} < \varepsilon$.
To do this, because $p \in L_2(\Omega)$, we first find a function $p_1$ with a compact support in $B_{R_1}$ such that
$$
\|p-p_1\|_{L_2(\Omega)} < \varepsilon/2.
$$
If $\int_{\Omega} p_1 \, dx = 0$, we set $R = R_1$ and $p_\varepsilon = p_1$.
Otherwise, set
$$
g = \frac{\varepsilon^2}{4 C} I_{B_{R_2}},
\quad
\text{where}
\quad
C = \int_\Omega p_1\, dx,
$$
and $R_2$ is a positive number satisfying
$$
|B_{R_2}| = 4\frac{C^2}{\varepsilon^2}
\quad
\text{if}
\quad
\Omega = \bR^d,
\quad
|B_{R_2}^+| = 4\frac{C^2}{\varepsilon^2}
\quad
\text{if}
\quad
\Omega = \bR^d_+.
$$
Then, we see that
$$
\int_{\Omega} (p_1- g) \, dx = 0\quad\text{and}
\quad
\int_{\Omega} g^2 \, dx = \varepsilon^2/4.
$$
Thus, it suffices to take $R = \max\{ R_1, R_2\}$ and $p_\varepsilon = p_1 - g$.

Thanks to the fact that $(p_{\varepsilon})_{B_R} = 0$ or $(p_{\varepsilon})_{B_R^+} = 0$, there exists a solution $\psi \in \mathring{W}_2^1(B_R)^d$ or $\psi \in \mathring{W}_2^1(B_R^+)^d$ satisfying the divergence equation
$$
\operatorname{div} \psi = p_\varepsilon
$$
in $B_R$ or $B_R^+$.
Extend $\psi$ to be zero on $\bR^d \setminus B_R$ or $\bR^d_+ \setminus B_R^+$.
Then we have $\psi \in \mathring{W}_2^1(\Omega)^d$ and
\begin{equation}
							\label{eq0427_01}
\|D\psi\|_{L_2(\Omega)} \le N(d) \|p_\varepsilon\|_{L_2(\Omega)}.
\end{equation}

By applying $\psi$ as a test function to the system, we have that
$$
\int_\Omega D_\alpha \psi \cdot A^{\alpha\beta} D_\beta u \, dx + \int_\Omega p \, p_\varepsilon \, dx = \int_\Omega f_\alpha \cdot D_\alpha \psi \, dx.
$$
From this, Young's inequality, and \eqref{eq0427_01}, it follows that
\begin{multline*}
\|p\|_{L_2(\Omega)}^2 \le \|p p_\varepsilon\|_{L_1(\Omega)}
+ \|p(p-p_\varepsilon)\|_{L_1(\Omega)}\\
\le \varepsilon_1 \|p_\varepsilon\|^2_{L_2(\Omega)} + \frac{N}{\varepsilon_1} \|f_\alpha\|_{L_2(\Omega)}^2 + \frac{N}{\varepsilon_1} \|Du\|_{L_2(\Omega)}^2
+ \varepsilon_1 \|p\|^2_{L_2(\Omega)}
+ \frac{N}{\varepsilon_1}\|p-p_\varepsilon\|^2_{L_2(\Omega)}
\end{multline*}
for any $\varepsilon_1 > 0$, where $N=N(d,\delta)$.
Since $\varepsilon > 0$ is arbitrary, from the above inequality we obtain that
$$
\|p\|_{L_2(\Omega)}^2 \le N \|f_\alpha\|_{L_2(\Omega)}^2 + N\|Du\|_{L_2(\Omega)}^2,
$$
which combined with \eqref{eq0427_02} proves the desired inequality.
\end{proof}

In the lemmas below, we do not necessarily have that
$(p)_{B_R} = 0$ or $(p)_{B_R^+} = 0$ unless specified.
We note that by now these lemmas are fairly standard results, and we present them here for the sake of completeness.
See, for instance, \cite{Ka04} and \cite{CL15}, and also \cite{MR0641818} under slightly different conditions on the coefficients.

\begin{lemma}
							\label{lem0225_2}
Let $R>0$.
If $(u,p) \in W_2^1(B_R)^d \times L_2(B_R)$ satisfies
\begin{equation}
							\label{eq0225_06}
\cL u + \nabla p = 0\quad\text{in}\,\,B_R,
\end{equation}
then
\begin{equation}
							\label{eq0225_03}
\int_{B_R} | p - (p)_{B_R} |^2 \, dx \le N \int_{B_R} |Du|^2 \, dx,
\end{equation}
where $N=N(d,\delta)$.
The same estimate holds if $B_R$ is replaced by $B_R^+$.
\end{lemma}

\begin{proof}
We only prove the case with $B_R^+$, because the other case is similar.
By Remark \ref{rem0229_1}, one can find $\psi \in \mathring{W}_2^1(B_R^+)^d$ satisfying
$$
\operatorname{div} \psi = p - (p)_{B_R^+}\quad \text{in}\,\,B_R^+
$$
and
\begin{equation}
							\label{eq0225_07}
\|D \psi\|_{L_2(B_R^+)}
\le N \| p - (p)_{B_R^+} \|_{L_2(B_R^+)},
\end{equation}
where $N = N(d)$.
Then, apply $\psi$ to \eqref{eq0225_06} as a test function, and use Young's inequality and \eqref{eq0225_07}, to get \eqref{eq0225_03} with $B_R^+$ in place of $B_R$.
\end{proof}

\begin{lemma}
							\label{lem0225_3}
Let $0 < r < R$.

\begin{enumerate}

\item
If $(u,p) \in W_2^1(B_R)^d \times L_2(B_R)$ satisfies
\begin{equation}
							\label{eq0224_01}
\left\{
\begin{aligned}
\cL u + \nabla p &= 0
\quad
\text{in}\,\,B_R,
\\
\operatorname{div} u &= 0
\quad
\text{in}\,\,B_R,
\end{aligned}
\right.
\end{equation}
then for any $\varepsilon > 0$, we have that
\begin{equation}
							\label{eq0225_05}
\int_{B_r} |Du|^2 \, dx \le N (R-r)^{-2} \int_{B_R}|u|^2 \, dx + \varepsilon \int_{B_R} |Du|^2 \, dx,
\end{equation}
where $N = N(d,\delta,\varepsilon)$.

\item
If $(u,p) \in W_2^1(B_R^+)^d \times L_2(B_R^+)$ satisfies
\begin{equation}
							\label{eq0225_01}
\left\{
\begin{aligned}
\cL u + \nabla p &= 0
\quad
\text{in}\,\,B_R^+,
\\
\operatorname{div} u &= 0
\quad
\text{in}\,\,B_R^+,
\quad
\\
u&=0 \quad
\text{on}\,\,
B_R \cap \partial\bR^d_+,
\end{aligned}
\right.
\end{equation}
then for any $\varepsilon>0$, we have that \eqref{eq0225_05} holds with $B_r^+$ and $B_R^+$ replacing $B_r$ and $B_R$, respectively, where $N = N(d,\delta, \varepsilon)$.
\end{enumerate}
\end{lemma}

\begin{proof}
We only prove the second assertion of the lemma, because the first is the same with obvious modifications.

Set $\eta$ to be an infinitely differentiable function on $\bR^d$, such that
\begin{equation}
							\label{eq0229_05}
0 \le \eta \le 1,
\quad
\eta = 1\,\,\text{on}\,\,B_r,
\quad
\eta =0 \,\,\text{on}\,\,\bR^d \setminus B_R,
\quad
|D\eta| \le N(d) (R-r)^{-1}.
\end{equation}
Then, we apply $\eta^2 u$  to \eqref{eq0225_01} as a test function (because $(\eta^2 u)|_{\partial B_R^+} = 0$), to obtain that
$$
\int_{B_R^+} D_\alpha\left(\eta^2 u\right) \cdot A^{\alpha\beta} D_\beta u \, dx + \int_{B_R^+} p\, \operatorname{div} (\eta^2 u) \, dx = 0.
$$
From this and the fact that $\int_{B_R^+} \operatorname{div} (\eta^2 u) \, dx = 0$, we have that
\begin{align*}
&\int_{B_R^+} \eta D_\alpha u \cdot A^{\alpha\beta} \eta D_\beta u \, dx \\
&= - 2 \int_{B_R^+}  (D_\alpha \eta) u \cdot A^{\alpha\beta} \eta D_\beta u \, dx - \int_{B_R^+} \left( p - (p)_{B_R^+} \right) \operatorname{div} (\eta^2 u) \, dx.
\end{align*}
Together with \eqref{eq0229_05}, the ellipticity condition, Young's inequality, and the fact that $\operatorname{div}u = 0$, this shows that
$$
\int_{B_R^+} \eta^2 |Du|^2 \, dx \le N(d,\delta,\varepsilon)(R-r)^{-2} \int_{B_R^+} |u|^2 \, dx + \varepsilon \int_{B_R^+} |p-(p)_{B_R^+}|^2 \, dx
$$
for any $\varepsilon > 0$.
The desired estimate follows by combining this with Lemma \ref{lem0225_2}, and the fact that $\eta = 1$ on $B_r$.
\end{proof}

\begin{lemma}
							\label{lem0225_4}
Let $0 < r < R$.
\begin{enumerate}
\item If $(u,p) \in W_2^1(B_R)^d \times L_2(B_R)$ satisfies \eqref{eq0224_01}, then we have that
\begin{equation}
							\label{eq0224_02}
\int_{B_r} |Du|^2 \, dx \le N (R-r)^{-2} \int_{B_R} |u|^2 \, dx,
\end{equation}
where $N=N(d,\delta)$.

\item
If $(u,p) \in W_2^1(B_R^+)^d \times L_2(B_R^+)$ satisfies \eqref{eq0225_01}, then we have that \eqref{eq0224_02} holds, with $B_r^+$ and $B_R^+$ replacing $B_r$ and $B_R$, respectively.
\end{enumerate}
\end{lemma}

\begin{proof}
To prove the lemma one may use the so called $\varepsilon$-lemma given in \cite[Lemma 0.5]{MR0641818}.
However, in this case, it is easier to employ the following well-known argument (see, for instance, the proof of \cite[Lemma 4.2]{MR2338417}).
Set
$$
R_0 = r,
\quad
R_k = r + (R-r)(1-2^{-k}),
\quad k = 1, 2, \ldots.
$$
Then, by Lemma \ref{lem0225_3} we have that
$$
\int_{B_{R_k}} |Du|^2 \, dx \le N  \frac{4^k}{(R -r)^2} \int_{B_{R_{k+1}}} |u|^2 \, dx + \varepsilon \int_{B_{R_{k+1}}} |Du|^2 \, dx,
\quad
k = 0, 1, 2, \ldots,
$$
for any $\varepsilon > 0$, where $N=N(d,\delta,\varepsilon)$.
By multiplying both sides of the above inequality by $\varepsilon^k$ and summing the terms with respect to $k = 0, 1, \ldots$, we obtain that
$$
\sum_{k=0}^\infty \varepsilon^k\int_{B_{R_k}} |Du|^2 \, dx \le \frac{N}{(R-r)^2} \sum_{k=0}^\infty (4\varepsilon)^k \int_{B_{R_{k+1}}} |u|^2 \, dx + \sum_{k=1}^\infty \varepsilon^k\int_{B_{R_k}} |Du|^2 \, dx,
$$
where each  summation is finite upon choosing, for instance, $\varepsilon = 1/8$.
Since the first summation on the right-hand side of the above inequality is bounded by $\int_{B_R}|u|^2 \, dx$, we can arrive at \eqref{eq0224_02} by subtracting $\sum_{k=1}^\infty \varepsilon^k \int_{B_{R_k}} |Du|^2 \, dx$ from both sides of the above inequality.
The other case for half balls is proved in the same way.
\end{proof}

\section{$L_\infty$ and H\"{o}lder estimates}							\label{sec_04}

In this section, we prove $L_\infty$ and H\"older estimates of certain linear combinations of $Du$ and $p$, which are crucial for proving our main results.
Recall the operator $\cL_0$ given in \eqref{eq0328_03}, where the coefficients are functions of $x_1$ only.
In this case, if a sufficiently smooth $(u,p)$ satisfies $\cL_0 u + \nabla p = 0$ in $\Omega \subset \bR^d$, we see that
\begin{equation}
							\label{eq0127_01}
D_1 \left( A^{1\beta} D_\beta u +
\left[
\begin{matrix}
p
\\
0
\\
\vdots
\\
0
\end{matrix}
\right]
\right)
= - \sum_{\alpha \ne 1} A^{\alpha\beta} D_{\alpha\beta} u -
\left[
\begin{matrix}
0
\\
D_2 p
\\
\vdots
\\
D_d p
\end{matrix}
\right]
\end{equation}
in $\Omega$.
Set $U = \left(U_1, U_2, \ldots, U_d\right)^{\operatorname{tr}}$, where
\begin{equation}
							\label{eq0229_10}
U_1 = \sum_{j=1}^d \sum_{\beta=1}^d A_{1j}^{1\beta} D_\beta u_j + p,
\quad
U_i = \sum_{j=1}^d \sum_{\beta=1}^d A_{ij}^{1\beta} D_\beta u_j,
\quad
i = 2, \ldots, d.
\end{equation}
That is,
$$
U = A^{1\beta}D_\beta u + (p,0,\ldots,0)^{\operatorname{tr}}.
$$

Here and throughout we write $D D_{x'}^k u$, $k=0,1, \ldots$, to denote $D^\vartheta u$,
where $\vartheta$ is a multi-index such that $\vartheta = (\vartheta_1, \ldots, \vartheta_d)$ with $\vartheta_1 = 0, 1$ and $|\vartheta| = k+1$.

\begin{lemma}
							\label{lem0301_1}
Let $0 < r < R$, and let $\ell$ be a constant.
\begin{enumerate}
\item If $(u,p) \in W_2^1(B_R)^d \times L_2(B_R)$ satisfies
\begin{equation}
							\label{eq0229_01}
\left\{
\begin{aligned}
 \cL_0 u + \nabla p &= 0
\quad
\text{in}\,\,B_R,
\\
\operatorname{div} u &= \ell
\quad
\text{in}\,\,B_R,
\end{aligned}
\right.
\end{equation}
then $DD_{x'} u \in L_2(B_r)$, and
\begin{equation}
							\label{eq0301_04}
\int_{B_r} |DD_{x'}u|^2 \, dx \le N(R-r)^{-2} \int_{B_R} |Du|^2 \, dx,
\end{equation}
where $N=N(d,\delta)$.

\item
If $(u,p) \in W_2^1(B_R^+)^d \times L_2(B_R^+)$ satisfies
\begin{equation}
							\label{eq0229_02}
\left\{
\begin{aligned}
\cL_0 u + \nabla p &= 0
\quad
\text{in}\,\,B_R^+,
\\
\operatorname{div} u &= \ell
\quad
\text{in}\,\,B_R^+,
\\
u &= 0
\quad
\text{on}\,\, B_R \cap \partial \bR^d_+,
\end{aligned}
\right.
\end{equation}
then $DD_{x'}u \in L_2(B_r^+)$, $D_{x'} u = 0$ on $B_r \cap \partial \bR^d_+$,
and
\eqref{eq0301_04} is satisfied with $B_r^+$ and $B_R^+$ replacing $B_r$ and $B_R$, respectively.
\end{enumerate}
\end{lemma}

\begin{proof}
We only deal with the second assertion here.
Set $\delta_{j,h} f$ to be the difference quotient of $f$ with respect to $x_j$, i.e.,
$$
\delta_{j,h} f(x) = \frac{f(x+e_j h) - f(x)}{h},
$$
and let $R_1 = (R+r)/2$.
Then, since the coefficients are functions of $x_1$ only, we have for $0 < h < (R-r)/2$ that
\begin{equation}
							\label{eq0301_03}
\left\{
\begin{aligned}
\cL_0 (\delta_{j,h} u) + \nabla (\delta_{j,h} p) &= 0 \quad \text{in} \,\, B_{R_1}^+,
\\
\operatorname{div} (\delta_{j,h} u) &= 0
\quad \text{in} \,\,B_{R_1}^+
\\
\delta_{j,h} u &= 0 \quad \text{on} \,\,
B_{R_1} \cap \partial \bR^d_+,
\end{aligned}
\right.
\end{equation}
where $j=2,\ldots,d$.
By applying Lemma \ref{lem0225_4} to \eqref{eq0301_03}, we have that
$$
\int_{B_r^+} |D (\delta_{j,h}u)|^2 \, dx \le N(R-r)^{-2} \int_{B_{R_1}^+} |\delta_{j,h}u|^2 \, dx,
$$
which we can combine with the standard finite difference argument to imply the desired conclusion.
\end{proof}

To estimate $U$, we also need to bound $D_{x'} p \in L_2(B_r)$, as in the following key lemma.

\begin{lemma}
                            \label{lem3.6}
Let $0 < r < R$, and let $\ell$ be a constant.

\begin{enumerate}

\item
If $(u,p) \in W_2^1(B_R)^d \times L_2(B_R)$ satisfies
\eqref{eq0229_01} in $B_R$, then $D_{x'}p \in L_2(B_r)$ and
\begin{equation}
							\label{eq0229_04}
\int_{B_r}|D_{x'}p|^2\,dx\le N(R-r)^{-2} \int_{B_R}|Du|^2\,dx,
\end{equation}
where $N = N(d,\delta)>0$.

\item If $(u,p) \in W_2^1(B_R^+)^d \times L_2(B_R^+)$ satisfies \eqref{eq0229_02} in $B_R^+$, then $D_{x'}p \in L_2(B_r^+)$, and \eqref{eq0229_04} is satisfied with $B_r^+$ and $B_R^+$ replacing $B_r$ and $B_R$, respectively.
\end{enumerate}
\end{lemma}

\begin{proof}
We only prove the second assertion here.
Define $\delta_{j,h} f$ as in the proof of Lemma \ref{lem0301_1}, and set $R_1 = (2R+r)/3$. Then, for $0 < h < (R-r)/3$, we have that
\begin{equation}
							\label{eq0301_02}
\cL_0 \left(\delta_{j,h} u \right) + \nabla \left(\delta_{j,h} p\right) = 0
\end{equation}
in $B_{R_1}^+$, where $j=2,\ldots,d$.
Set $R_2 = (R+2r)/3$, and let $\eta$ be an infinitely differentiable function on $\bR^d$ such that
\begin{align*}
&0 \le \eta \le 1,
\quad
\eta = 1\,\,\text{on}\,\,B_r,
\quad
\eta =0 \,\,\text{on}\,\,\bR^d \setminus B_{R_2},\\
&\text{and}\quad |D\eta| \le N(d)(R_2 - r)^{-1} = N(d) (R-r)^{-1}.
\end{align*}
Find a function $\psi \in \mathring{W}_2^1(B_{R_1}^+)^d$ such that
$$
\operatorname{div} \psi = \delta_{j,h} \left( (p - c)\eta^2 \right)
$$
in $B_{R_1}^+$,
where $c := (p)_{B_R^+}$.
Note that
$$
\int_{B_{R_1}^+}\delta_{j,h} \left( (p - c)\eta^2 \right) \, dx = 0.
$$
Then, by Remark \ref{rem0229_1} we have that
$$
\|D\psi\|_{L_2(B_{R_1}^+)} \le N(d) \| \delta_{j,h} \left( (p - c)\eta^2 \right) \|_{L_2(B_{R_1}^+)}.
$$
Since
$$
\delta_{j,h} \left( (p - c)\eta^2 \right)
= \eta^2(x) ( \delta_{j,h}p )(x) + (p - c)(x+h) (\delta_{j,h} \eta^2)(x),
$$
it follows that
\begin{equation}
							\label{eq0302_01}
\|D\psi\|_{L_2(B_{R_1}^+)} \le N \|(\delta_{j,h}p) \eta\|_{L_2(B_{R_1}^+)} + N(R-r)^{-1}\|p-c\|_{L_2(B_R^+)}.
\end{equation}
Then, applying $\psi$ to \eqref{eq0301_02} as a test function, we have that
$$
\int_{B_{R_1}^+} \left(\delta_{j,h} p\right) \delta_{j,h}  \left( (p-c)\eta^2 \right) \, dx = - \int_{B_{R_1}^+} D_\alpha \psi \cdot A^{\alpha\beta} D_{\beta}\left(\delta_{j,h} u \right) \, dx.
$$
Thus, we have that
\begin{align*}
&\int_{B_{R_1}^+} \left(\delta_{j,h} p\right)^2  \eta^2 \, dx
= - \int_{B_{R_1}^+} D_\alpha \psi \cdot A^{\alpha\beta} D_{\beta}\left(\delta_{j,h} u \right) \, dx\\
&\quad - \int_{B_{R_1}^+} (\delta_{j,h} p)(x) (p-c)(x+h) (\delta_{j,h} \eta)(x)  \left( \eta(x+h) + \eta(x) \right) \, dx.
\end{align*}
By Young's inequality and \eqref{eq0302_01}, we have for any $\varepsilon \in (0,1)$ that
\begin{multline}
							\label{eq0302_02}
\|  (\delta_{j,h} p)\eta \|_{L_2(B_{R_1}^+)}^2 \\
\le \varepsilon\| D\psi\|_{L_2(B_{R_1}^+)}^2 + N(\varepsilon, \delta) \| D (\delta_{j,h} u)\|_{L_2(B_{R_1}^+)}^2
+ \varepsilon \| (\delta_{j,h} p) \eta  \|_{L_2(B_{R_1}^+)}^2
\\
+ \varepsilon \| (\delta_{j,h} p) \eta(\cdot+h)  \|_{L_2(B_{R_1}^+)}^2
+ N(\varepsilon) \| (p-c)(\cdot+h) (\delta_{j,h} \eta) \|_{L_2(B_{R_1}^+)}^2.
\end{multline}
Here, we note that
$$
\| (\delta_{j,h} p) \eta(\cdot+h)  \|_{L_2(B_{R_1}^+)}
\le \| (\delta_{j,h} p) \eta \|_{L_2(B_{R_1}^+)} + \| (\delta_{j,h} p) \left(\eta(\cdot+h) - \eta(\cdot)\right) \|_{L_2(B_{R_1}^+)},
$$
where the last term is estimated by
$$
\| (\delta_{j,h} p) \left(\eta(\cdot+h) - \eta(\cdot)\right) \|_{L_2(B_{R_1}^+)}^2 =
\int_{B_{R_1}^+} \left| \left( p(x+h) - p(x) \right) (\delta_{j,h} \eta)(x) \right|^2 \, dx
$$
$$
\le N(R-r)^{-2}\| p - c \|_{L_2(B_R^+)}^2.
$$
In addition, note that
$$
\| (p-c)(\cdot+h) (\delta_{j,h} \eta) \|_{L_2(B_{R_1}^+)}^2
\le N(R-r)^{-2} \|p-c\|_{L_2(B_R^+)}^2,
$$
and by Lemma \ref{lem0301_1} and the properties of $\delta_{j,h}$, it holds that
$$
\| D (\delta_{j,h} u)\|_{L_2(B_{R_1}^+)} \le \|D D_{x'} u\|_{L_2(B_{R'}^+)}
\le N (R-r)^{-2} \|Du\|_{L_2(B_R^+)},
$$
where $R_1 < R' < R$.
By using the above inequalities combined with \eqref{eq0302_02}, \eqref{eq0302_01}, and Lemma \ref{lem0225_2}, and choosing a sufficiently small $\varepsilon >0$, we obtain that
$$
\| (\delta_{j,h} p) \eta \|_{L_2(B_{R_1}^+)} \le N(d,\delta) (R-r)^{-2} \|Du\|_{L_2(B_R^+)}.
$$
Together with the properties of the finite difference operator, this proves the desired inequality.
\end{proof}

As usual, by $[u]_{C^\tau(\Omega)}$, $\tau \in (0,1)$, we denote the H\"{o}lder semi-norm of $u$ defined by
$$
[u]_{C^\tau(\Omega)} = \sup_{\substack{x,y \in \Omega \\ x \ne y}} \frac{|u(x) - u(y)|}{|x-y|^\tau}.
$$

The following interior and boundary $L_\infty$ and H\"older estimates constitute the main results of this section.

\begin{lemma}
							\label{lem0229_1}
Let $\ell$ be a constant, and let $(u,p) \in W_2^1(B_2)^d \times L_2(B_2)$ satisfy \eqref{eq0229_01} with $R=2$.
Then, we have that
\begin{equation*}
\|D_{x'}u\|_{L_\infty(B_1)} + \left[D_{x'}u\right]_{C^{1/2}(B_1)} + \left[U\right]_{C^{1/2}(B_1)}
\le N\|Du\|_{L_2(B_2)},
\end{equation*}
and
$$
\|U_i\|_{L_\infty(B_1)} \le N\|Du\|_{L_2(B_2)},
\quad
i = 2, \ldots, d,
$$
where $N=N(d,\delta)$.
\end{lemma}

\begin{proof}
See the proof of Lemma \ref{lem0229_2} below, with obvious modifications.
\end{proof}

\begin{lemma}
							\label{lem0229_2}
Let $\ell$ be a constant, and let $(u,p) \in W_2^1(B_2^+)^d \times L_2(B_2^+)$ satisfy \eqref{eq0229_02} with $R=2$.
Then, we have that
\begin{equation*}
\|D_{x'}u\|_{L_\infty(B_1^+)} + \left[D_{x'}u \right]_{C^{1/2}(B_1^+)} + \left[U\right]_{C^{1/2}(B_1^+)}
\le N \|Du\|_{L_2(B_2^+)},
\end{equation*}
and
$$
\|U_i\|_{L_\infty(B_1^+)} \le N\|Du\|_{L_2(B_2^+)},
\quad
i = 2, \ldots, d,
$$
where $N=N(d,\delta)$.
\end{lemma}

\begin{proof}
From Lemmas \ref{lem0301_1} and \ref{lem3.6}, we have that $(D_{x'} u, D_{x'}p) \in W_2^1(B_{r_1}^+)^d \times L_2(B_{r_1}^+)$ and
\begin{equation*}
\|D D_{x'} u \|_{L_2(B_{r_1}^+)}^2 + \| D_{x'}p \|_{L_2(B_{r_1}^+)}^2 \le N(d,\delta, r_1) \int_{B_2^+} |Du|^2 \, dx,
\end{equation*}
where $1 < r_1 < 2$.
Moreover, $(D_{x'} u, D_{x'} p)$ satisfies
$$
\left\{
\begin{aligned}
\cL_0 (D_{x'} u) + \nabla (D_{x'} p) &= 0
\quad
\text{in}\,\,B_{r_1}^+,
\\
\operatorname{div} (D_{x'}u) &= 0
\quad
\text{in}\,\,B_{r_1}^+,
\\
D_{x'} u &= 0
\quad
\text{on}\,\, B_{r_1} \cap \partial \bR^d_+.
\end{aligned}
\right.
$$
Then, we apply Lemmas \ref{lem0301_1} and \ref{lem3.6} again as above, with $r_2$ in place of $r_1$ and with $r_1$ in place of $1$, where $1 < r_2 < r_1 < 2$.
By repeating this process, we see that
$(D_{x'}^k u, D_{x'}^k p)$ belongs to $W_2^1(B_r^+)^d \times L_2(B_r^+)$ with $D_{x'}^k u = 0$ on $B_r \cap \partial \bR^d_+$,
and satisfies
\begin{equation}
							\label{eq0229_06}
\int_{B_r^+}|DD^k_{x'}u|^2 \, dx + \int_{B_r^+}|D^k_{x'} p|^2 \, dx \le N(d,\delta,r,k) \int_{B_2^+} |D u|^2 \, dx
\end{equation}
for any $r \in [1,2)$ and $k=1,2,\ldots$.
In particular, this estimate means that $D_{x'}u$ has one derivative in $x_1$ and sufficiently many derivatives in $x_i$, $i=2,\cdots,d$, the $L_2(B_1^+)$ norms of which are bounded by $\|Du\|_{L_2(B_2^+)}$. Then, by the anisotropic Sobolev embedding theorem with $k> (d-1)/2$ (see, for instance, the proof \cite[Lemma 3.5]{MR2800569}), we get that
$$
\|D_{x'}u\|_{L_\infty(B_1^+)} + [D_{x'}u]_{C^{1/2}(B_1^+)} \le N(d,\delta) \|Du\|_{L_2(B_2^+)}.
$$

Now, we prove the H\"{o}lder semi-norm estimate of $U$ and the sup-norm estimate of $U_i$, $i = 2, \ldots, d$.
Set
$$
\tilde{U} = A^{1\beta}D_\beta u + \left( p - (p)_{B_1^+}, 0, \ldots, 0 \right)^{\operatorname{tr}}.
$$
By the definitions of $U$ and $\tilde{U}$, we have that
$$
\|D_{x'}^k \tilde{U}\|_{L_2(B_1^+)} = \|D_{x'}^k U\|_{L_2(B_1^+)} \le N(d,\delta)\| D D_{x'}^k u\|_{L_2(B_1^+)} + \|D_{x'}^k p\|_{L_2(B_1^+)},
$$
where $k=1,2,\ldots$.
In combination with \eqref{eq0229_06}, this shows that
\begin{equation}
							\label{eq0229_08}
\|D_{x'}^k \tilde{U}\|_{L_2(B_1^+)} \le N(d,\delta,k) \|Du\|_{L_2(B_2^+)}.
\end{equation}
Similarly, since the $D_{\alpha\beta}u$ terms on the right-hand side of \eqref{eq0127_01} are of the form $D D_{x'} u$, we have that
\begin{equation}
							\label{eq0229_09}
\|D_1 D_{x'}^k \tilde{U}\|_{L_2(B_1^+)} = \|D_1 D_{x'}^k U\|_{L_2(B_1^+)}
\le N(d,\delta,k)\|Du\|_{L_2(B_2^+)},
\end{equation}
where $k=0,1,2,\ldots$.
To estimate $\|\tilde{U}\|_{L_2(B_1^+)}$, we apply Lemma \ref{lem0225_2} with $R=1$ to obtain that
$$
\|p-(p)_{B_1^+}\|_{L_2(B_1^+)} \le N \|Du\|_{L_2(B_1^+)} \le N \|Du\|_{L_2(B_2^+)},
$$
where $N=N(d,\delta)$.
Together with the definition of $\tilde{U}$, this shows that
$$
\|\tilde{U}\|_{L_2(B_1^+)} \le N(d,\delta) \|Du\|_{L_2(B_2^+)}.
$$
Together with \eqref{eq0229_08} and \eqref{eq0229_09}, and using the anisotropic Sobolev embedding as above with $k>(d-1)/2$, this gives that
$$
\|\tilde{U}\|_{L_\infty(B_1^+)}+ \big[\tilde{U}\big]_{C^{1/2}(B_1^+)} \le N(d,\delta)\|Du\|_{L_2(B_2^+)}.
$$
Since $\left[U\right]_{C^{1/2}(B_1^+)} = \big[\tilde{U}\big]_{C^{1/2}(B_1^+)}$ and $\tilde{U}_i = U_i$, $i = 2, \ldots, d$, we have obtained the desired inequalities. Thus, the lemma is proved.
\end{proof}

\section{Mean oscillation estimates}
                                    \label{sec_05}
In this section, we prove our mean oscillation estimates using the H\"{o}lder estimates developed in Section \ref{sec_04} and the $L_2$-estimates of the Stokes system given in Lemma \ref{lem0225_1}.
Throughout this section, we consider the operator $\cL_0$, i.e., the coefficients $A^{\alpha\beta}$ are measurable functions of $x_1$ only.

\begin{lemma}
							\label{lem0311_1}
Let $r \in (0,\infty)$, $\kappa \ge 2$, $x_0 \in \bR^d$, and $f_\alpha, g \in L_2(B_{\kappa r}(x_0))$.
If $(u,p) \in W_2^1(B_{\kappa r}(x_0))^d \times L_2(B_{\kappa r}(x_0))$ satisfies
$$
\begin{cases}
\cL_0 u + \nabla p = D_\alpha f_\alpha
\quad
&\text{in}\,\,B_{\kappa r}(x_0),
\\
\operatorname{div} u = g
\quad
&\text{in}\,\,B_{\kappa r}(x_0),
\end{cases}
$$
then
$$
\left( |D_{x'}u - (D_{x'}u)_{B_r(x_0)}| \right)_{B_r(x_0)} + \left( |U - (U)_{B_r(x_0)}| \right)_{B_r(x_0)}
$$
$$
\le N \kappa^{-1/2} \left( |Du|^2 \right)^{1/2}_{B_{\kappa r}(x_0)}
+ N \kappa^{d/2} \left(|f_\alpha|^2 + |g|^2  \right)^{1/2}_{B_{\kappa r}(x_0)},
$$
where $N=N(d,\delta)$.
\end{lemma}

\begin{proof}
This lemma follows as a consequence of Lemmas \ref{lem0225_1} and \ref{lem0229_1}.
See Case 2 in the proof of Lemma \ref{lem0302_1} below.
\end{proof}

\begin{lemma}
							\label{lem0302_1}
Let $r \in (0,\infty)$, $\kappa \ge 16$, $x_0 \in \overline{\bR^d_+}$, and $ f_\alpha, g \in L_2(B^+_{\kappa r}(x_0))$.
If $(u,p) \in W_2^1(B_{\kappa r}^+(x_0))^d \times L_2(B_{\kappa r}^+(x_0))$ satisfies
$$
\begin{cases}
\cL_0 u + \nabla p = D_\alpha f_\alpha
\quad
&\text{in}\,\,B_{\kappa r}^+(x_0),
\\
\operatorname{div} u = g
\quad
&\text{in}\,\,B_{\kappa r}^+(x_0),
\\
u = 0
\quad
&\text{on} \,\, B_{\kappa r} \cap \partial \bR^d_+,
\end{cases}
$$
then
\begin{multline}
							\label{eq0311_01}
\left( |D_{x'}u - (D_{x'}u)_{B_r^+(x_0)}| \right)_{B_r^+(x_0)} + \left( |U - (U)_{B_r^+(x_0)}| \right)_{B_r^+(x_0)}
\\
\le N \kappa^{-1/2} \left( |Du|^2 \right)^{1/2}_{B_{\kappa r}^+(x_0)}
+ N \kappa^{d/2}
\left( |f_\alpha|^2 + |g|^2  \right)^{1/2}_{B_{\kappa r}^+(x_0)},
\end{multline}
where $N=N(d,\delta)$.
\end{lemma}

\begin{proof}
Denote the first coordinate of $x_0$ by ${x_0}_1$.
We consider the following two cases.

\noindent{\bf Case 1}: ${x_0}_1 \ge \kappa r/8$.
In this case, we have that
$$
B_r^+(x_0) = B_r(x_0) \subset B_{\kappa r/8}(x_0) \subset \bR^d_+
$$
and $\kappa/8\ge 2$.
Then, the estimate \eqref{eq0311_01} follows from Lemma \ref{lem0311_1}.

\noindent{\bf Case 2}: ${x_0}_1 < \kappa r/8$.
Set $y_0 = (0,x_0')$. Then, we have that
\begin{equation}
                                \label{eq4.27}
B_r^+(x_0) \subset B_{\kappa r/4}^+(y_0)
\subset B_{\kappa r/2}^+(y_0) \subset B_{\kappa r}^+(x_0).
\end{equation}
Considering dilation, it suffices to prove \eqref{eq0311_01} when $r = 4/\kappa \le 1/4$ and ${x_0}_1 < 1/2$.
Furthermore, we assume that $y_0 = 0$.
By Lemma \ref{lem0225_1}, there exists $(w,p_1) \in \mathring{W}_2^1(B_2^+)^d \times L_2(B_2^+)$ such that $(p_1)_{B_2^+} = 0$,
$$
\begin{cases}
\cL_0 w + \nabla p_1 =D_\alpha f_\alpha
\quad
&\text{in}\,\,B_2^+,
\\
\operatorname{div} w = g - (g)_{B_2^+}
\quad
&\text{in}\,\,B_2^+,
\end{cases}
$$
and
\begin{equation}
							\label{eq0225_02}
\|Dw\|_{L_2(B_2^+)} + \|p_1\|_{L_2(B_2^+)} \le N \left( \|f_\alpha\|_{L_2(B_2^+)} + \|g\|_{L_2(B_2^+)}\right),
\end{equation}
where $N=N(d,\delta)$.
In particular, we have that $w = 0$ on $B_2 \cap \partial \bR^d_+$.
The estimate \eqref{eq0225_02} clearly implies that
\begin{equation}
							\label{eq0302_04}
\|Dw\|_{L_2(B_r^+(x_0))} + \|p_1\|_{L_2(B_r^+(x_0))} \le N \left( \|f_\alpha\|_{L_2(B_2^+)} + \|g\|_{L_2(B_2^+)}\right),
\end{equation}
where $N=N(d,\delta)$.

Now, we set $(v,p_2) = (u,p) - (w,p_1)$, which satisfies
$$
\begin{cases}
\cL_0 v + \nabla p_2 = 0
\quad
&\text{in}\,\,B_2^+,
\\
\operatorname{div} v = (g)_{B_2^+}
\quad
&\text{in}\,\,B_2^+,
\\
v = 0
\quad
&\text{on} \,\, B_2 \cap \partial \bR^d_+.
\end{cases}
$$
Then, by Lemma \ref{lem0229_2},
\begin{align*}
&\left( |D_{x'} v - (D_{x'} v)_{B_r^+(x_0)}| \right)_{B_r^+(x_0)}
\le (2r)^{1/2} \left[D_{x'}v\right]_{C^{1/2}(B_r^+(x_0))}\\
&\,\,\le (2r)^{1/2} \left[D_{x'}v\right]_{C^{1/2}(B_1^+)}
\le N \kappa^{-1/2} \|Dv\|_{L_2(B_2^+)},
\end{align*}
where $N=N(d,\delta)$.
Similarly, we have that
$$
\left( |V - (V)_{B_r^+(x_0)}| \right)_{B_r^+(x_0)}
\le N \kappa^{-1/2} \|Dv\|_{L_2(B_2^+)},
$$
where $V$ is defined in exactly the same way as $U$ in \eqref{eq0229_10} with $v$ in place of $u$.
Then, it follows from the triangle inequality that
\begin{align*}
&\left( |D_{x'} u - (D_{x'} u)_{B_r^+(x_0)}| \right)_{B_r^+(x_0)}\\
&\le \left( |D_{x'} v - (D_{x'} v)_{B_r^+(x_0)}| \right)_{B_r^+(x_0)} + 2 \left( |D_{x'} w| \right)_{B_r^+(x_0)}\\
&\le N \kappa^{-1/2} \left( |Dv|^2 \right)_{B_{2}^+}^{1/2} + N\kappa^{d/2} \|D_{x'}w\|_{L_2(B_r^+)},
\end{align*}
where $N=N(d,\delta)$.
Together with the estimates \eqref{eq0225_02} and \eqref{eq0302_04}, and the fact that $u = v+w$, this shows that
\begin{align*}
&\left( |D_{x'} u - (D_{x'} u)_{B_r^+(x_0)}| \right)_{B_r^+(x_0)}\\
& \le N \kappa^{-1/2} \left( |Du|^2 \right)_{B_{2}^+}^{1/2} + N \kappa^{d/2} \left(|f_\alpha|^2 + |g|^2 \right)_{B_{2}^+}^{1/2}.
\end{align*}
It only remains to observe that the right-hand side is bounded by that of \eqref{eq0311_01}, because of \eqref{eq4.27}.

We can similarly obtain the desired estimate for $U$. Thus, the lemma is proved.
\end{proof}

\section{Proof of Theorem \ref{thm0307_1}}
                                    \label{sec_06}
In this section, we complete the proof of Theorem \ref{thm0307_1}.
We use the following filtration of partitions of $\bR^d$:
$$
\bC_n := \{ C_n = C_n(i_1, \ldots, i_d): (i_1, \ldots, i_d) \in \bZ^d \},
$$
where $n \in \bZ$ and
$$
C_n(i_1, \ldots, i_d) = [i_1 2^{-n}, (i_1 + 1) 2^{-n}) \times \cdots \times [i_d 2^{-n}, (i_d + 1) 2^{-n}).
$$
For a filtration of partitions of $\bR^d_+$, we replace $i_1 \in \bZ$ by $i_1 \in \{0,1,2,\ldots,\}$.
Using these filtrations, we define the sharp function of $f \in L_{1, \operatorname{loc}}(\Omega)$, where $\Omega = \bR^d$ or $\Omega = \bR^d_+$, by
$$
f^{\#}(x) = \sup_{n < \infty} \dashint_{C_n \ni x} \left| f(y) - (f)_{C_n} \right| \, dy,
$$
where the supremum is taken with respect to all $C_n \in \bC_n$ containing $x$, where $n \in \bZ$.
The maximal function of $f$ in $\bR^d$ or $\bR^+$ is defined by
\begin{equation}
							\label{eq0328_01}
\cM f(x) = \sup_{x_0\in \bar\Omega, \Omega \cap B_r(x_0) \ni x} \dashint_{\Omega \cap B_r(x_0)} |f(y)| \, dy,
\end{equation}
where $\Omega = \bR^d$ or $\Omega = \bR^d_+$, and the supremum is taken with respect to all $B_r(x_0)$ containing $x$ with $r > 0$, where $x_0\in \bar\Omega$.

\begin{proof}[Proof of Theorem \ref{thm0307_1}]
Because Theorem \ref{thm0427} covers the case with $q=2$, we assume that $q \in (2,\infty)$.
We prove the case when $\Omega = \bR^d_+$. The other case is simpler.

For $x \in \bR^d_+$ and $C_n \in \bC_n$ such that $x \in C_n$, find $x_0 \in \bR^d_+$ and the smallest $r \in (0,\infty)$ (indeed, $r = 2^{-n-1}\sqrt{d}$) satisfying
$C_n \subset B_r(x_0)$ and
\begin{equation}
							\label{eq0311_03}
\dashint_{C_n} | h(x) - (h)_{C_n} | \, dx \le N(d) \dashint_{B_r(x_0)}| h(x) - (h)_{B_r(x_0)} | \, dx.
\end{equation}
Since $(u,p) \in W_2^1(B_{\kappa r}^+(x_0))^d \times L_2(B_{\kappa r}^+(x_0))$, it follows from Lemma \ref{lem0302_1} that we have the mean oscillation estimate \eqref{eq0311_01} for $\kappa \ge 16$.
Moreover, each term in the right-hand side of \eqref{eq0311_01} is bounded by its maximal function at $x$.
From this and \eqref{eq0311_03}, we have that
\begin{multline*}
\left( |D_{x'}u - (D_{x'}u)_{C_n}| \right)_{C_n} + \left( |U - (U)_{C_n}| \right)_{C_n}
\\
\le N \kappa^{-1/2} \left( \cM (|Du|^2) (x) \right)^{1/2}
+ N \kappa^{d/2}
\left( \cM (|f_\alpha|^2)(x) \right)^{1/2} +N \kappa^{d/2} \left( \cM(|g|^2)(x)\right)^{1/2}
\end{multline*}
for $x \in C_n$ and $\kappa \ge 16$, where $N=N(d,\delta)$.
By taking the supremum of the left-hand side of the above inequality with respect to all $C_n \ni x$, $n \in \bZ$, we obtain that
\begin{align*}
&\left( D_{x'}u \right)^{\#}(x) + U^{\#}(x)
\le N \kappa^{-1/2} \left( \cM (|Du|^2) (x) \right)^{1/2}\\
&\quad + N \kappa^{d/2}
\left( \cM (|f_\alpha|^2)(x) \right)^{1/2} + N \kappa^{d/2} \left( \cM(|g|^2)(x)\right)^{1/2}
\end{align*}
for $x \in \bR^d_+$ and $\kappa \ge 16$.
Then, we employ the Fefferman-Stein theorem on sharp functions (see, for instance, \cite[Theorem 3.2.10]{MR2435520}) and the maximal function theorem (see, for instance, \cite[Theorem 3.3.2]{MR2435520} or Lemma \ref{lem7.2} in this paper, which also holds when $\Omega = \bR^d_+$ with $N=N(d,q)$) on the above pointwise estimate, to obtain that
\begin{equation}
							\label{eq0311_04}
\|D_{x'}u\|_{L_q} + \|U\|_{L_q} \le N \kappa^{-1/2} \|Du\|_{L_q} + N \kappa^{d/2} \left( \|f_\alpha\|_{L_q} + \|g\|_{L_q} \right),
\end{equation}
where $L_q = L_q(\bR^d_+)$ and $N = N(d,\delta,p)$.
Note that on the left-hand side of the above inequality we do not yet have $L_q$-norms of $D_1u_i$, $i=1,2,\ldots,d$, and $p$.
To obtain $L_q$-estimates of such terms,
we first note the relation
$$
D_1 u_1 + \ldots + D_d u_d = g.
$$
Using this and \eqref{eq0311_04}, we have that
\begin{equation}
							\label{eq0311_05}
\|D_1u_1\|_{L_q} + \|D_{x'}u\|_{L_q} + \|U\|_{L_q} \le N \kappa^{-1/2} \|Du\|_{L_q} + N \kappa^{d/2} \left( \|f_\alpha\|_{L_q} + \|g\|_{L_q} \right).
\end{equation}
Then, we use the relation
\begin{equation}
							\label{eq0127_04}
\sum_{j=2}^d A_{ij}^{11}D_1 u_j = U_i - \sum_{j=1}^d \sum_{\beta=2}^d A_{ij}^{1\beta} D_\beta u_j - A_{i1}^{11}D_1 u_1,
\quad i = 2, \ldots, d,
\end{equation}
which follows from the definition of $U_i$, $i=2,\ldots,d$.
By the ellipticity condition on $A^{\alpha\beta}$, it follows that the $(d-1)\times(d-1)$ matrix $[A_{ij}^{11}]_{i,j=2}^d$ is invertible.
Thus, from \eqref{eq0127_04} and \eqref{eq0311_05} we have that
$$
\|Du\|_{L_q} + \|U\|_{L_q} \le N \kappa^{-1/2} \|Du\|_{L_q} + N \kappa^{d/2} \left( \|f_\alpha\|_{L_q} + \|g\|_{L_q} \right).
$$
Upon taking a sufficiently large $\kappa \ge 16$, which depends only on $d$, $\delta$, and $q$, such that $N \kappa^{-1/2} \le 1/2$, we arrive at
\begin{equation}
							\label{eq0311_06}
\|Du\|_{L_q} + \|U\|_{L_q} \le N \left( \|f_\alpha\|_{L_q} + \|g\|_{L_q} \right).
\end{equation}
Finally, from this estimate and the definition of $U_1$, we see that the $L_q$-norm of $p$ is bounded by the right-hand side of \eqref{eq0307_03}.
By this and \eqref{eq0311_06}, we can conclude that the estimate \eqref{eq0307_03} holds, and the theorem is proved.
\end{proof}

\section{Proof of Theorem \ref{thm0307_2}}
                                    \label{sec_07}

This section is devoted to the proof of  Theorem \ref{thm0307_2}.
For any $x_0\in \bR^d$ and $r>0$, denote
$$
\Omega_r(x_0) = \Omega \cap B_r(x_0).
$$
We first derive the following mean oscillation estimate.

\begin{lemma}
                                        \label{lem6.5}
Let $\mu, \nu \in (1,\infty)$ be such that $1/\mu + 1/\nu = 1$ and $\kappa\ge 32$.
Then, under Assumption \ref{assum0307_1} ($\gamma,\rho$) such that $\rho\kappa\le 1/4$, for any $r \in (0, R_1/\kappa]$, $x_0 \in \overline{\Omega}$, and
$$
(u,p) \in W^1_{2\mu}(\Omega_{\kappa r}(x_0))^d \times L_{2}(\Omega_{\kappa r}(x_0))
$$
satisfying
\begin{equation}
                            \label{eq11.33}
\begin{cases}
\cL u + \nabla p = D_\alpha f_\alpha
\quad
&\text{in}\,\,\Omega_{\kappa r}(x_0),
\\
\operatorname{div} u = g
\quad
&\text{in}\,\,\Omega_{\kappa r}(x_0),
\\
u = 0
\quad
&\text{on} \,\, \partial \Omega\cap B_{\kappa r}(x_0),
\end{cases}
\end{equation}
 where $f_\alpha \in L_{2}(\Omega_{\kappa r}(x_0))$,
there exists a $d^2$-dimensional vector-valued function $\cU$ on $\Omega_{\kappa r}(x_0)$ such that on $\Omega_{\kappa r}(x_0)$,
\begin{equation}
                                    \label{eq10.54}
N^{-1} |Du| \le |\cU| \le N |Du|
\end{equation}
and
\begin{align}
&\left( |\cU - (\cU)_{\Omega_r(x_0)}|\right)_{\Omega_r(x_0)}
\le N (\kappa^{-\frac 1 2} + \kappa \rho) \left( |Du|^2 \right)_{\Omega_{\kappa r}(x_0)}^{\frac{1}{2}}\nonumber\\
                                    \label{eq3.33}
&\,\,+ N \kappa^{\frac d 2} \left(f_\alpha^2\right)_{\Omega_{\kappa r}(x_0)}^{\frac{1}{2}}+N \kappa^{\frac d 2} \left(g^2\right)_{\Omega_{\kappa r}(x_0)}^{\frac{1}{2}}
+ N \kappa^{\frac d 2} (\gamma+\rho)^{\frac{1}{2\nu}} \left( |Du|^{2\mu} \right)_{\Omega_{\kappa r}(x_0)}^{\frac{1}{2\mu}},
\end{align}
where $N=N(d,\delta, \mu)$.
\end{lemma}

\begin{proof}
We mainly follow the proof of Proposition 7.10 in \cite{MR2835999}, where $\bR^d_+$ instead of $\Omega$ is considered.
Let $\tilde x\in \partial\Omega$ be such that $|x_0 -\tilde x|=\text{dist}(x_0,\partial\Omega)$.
As in the proof of Lemma \ref{lem0302_1}, we consider two cases.

\noindent{\bf Case 1: $|x_0 -\tilde x|\ge \kappa r/16$.} In this case, we have that
$$
\Omega_r(x_0)=B_r(x_0)\subset B_{\kappa r/16}(x_0)\subset \Omega.
$$
Since $\kappa/16\ge 2$, \eqref{eq3.33} follows from Lemma \ref{lem0311_1}, by using a rotation of coordinates and setting
$$
\cU=(D_{x'}u,\operatorname{div} u,U_2,\ldots,U_d),
$$
where for $i = 2, \ldots,d$, $U_i$ are given as in \eqref{eq0328_05} below.
See the proof for Case 2.
As in the proof of Theorem \ref{thm0307_1}, by using the definition of $U$, we see that \eqref{eq10.54} is satisfied.

\noindent{\bf Case 2: $|x_0 -\tilde x|< \kappa r/16$.} Without loss of generality, one may assume that $\tilde x$ is the origin. Note that
\begin{equation*}
\Omega_r(x_0)\subset \Omega_{\kappa r/4} \subset \Omega_{\kappa r/2}
\subset \Omega_{\kappa r}(x_0).
\end{equation*}
Denote $R=\kappa r/2 \, (\le R_1/2)$.
Due to Assumption \ref{assum0307_1}, we can take an orthogonal transformation to obtain that
$$
 \{(x_1, x'):\rho R< x_1\}\cap B_R
 \subset\Omega\cap B_R
 \subset \{(x_1, x'):-\rho R<x_1\}\cap B_R
$$
and
\begin{equation}
                                \label{eq17_50}
\dashint_{B_R} \left| A^{\alpha\beta}(x_1,x') - \bar A^{\alpha\beta}(x_1)\right| \, dx \le \gamma,
\end{equation}
where
\begin{equation}
							\label{eq0715_01}
\bar A^{\alpha\beta}(x_1)=\dashint_{B'_R} A^{\alpha\beta}(x_1,x')\,dx'.
\end{equation}
Take a smooth function $\chi$ on $\bR$ such that
$$
\chi(x_1)\equiv 0\quad\text{for}\,\,x_1\le \rho R,
\quad \chi(x_1)\equiv 1\quad\text{for}\,\,x_1\ge  2\rho R,\quad
\text{and}
\quad
|\chi'| \le N(\rho R)^{-1}.
$$
Denote $\cL_0$ to be the elliptic operator with the coefficients $\bar A^{\alpha\beta}$ from \eqref{eq0715_01}.
Let $\hat u=\chi u$, which vanishes on $B_R\cap \{x_1\le \rho R\}$. From \eqref{eq11.33}, it is easily seen that $(\hat u,p)$ satisfies
\begin{equation}
                                    \label{eq17.23b}
\begin{cases}
\cL_0 \hat u+\nabla p = D_\alpha (\tilde f_\alpha+h_\alpha)
\quad
&\text{in}\,\,B_{R}\cap\{x_1>\rho R\},
\\
\operatorname{div}  \hat u = \chi g+\chi' u_1
\quad
&\text{in}\,\,B_{R}\cap\{x_1>\rho R\},
\\
\hat u = 0
\quad
&\text{on} \,\, B_{R}\cap \{x_1=\rho R\},
\end{cases}
\end{equation}
where
$$
\tilde f_\alpha=f_\alpha+(\bar A^{\alpha\beta}- A^{\alpha\beta})D_\beta  u\quad
\text{and}
\quad
h_\alpha=\bar A^{\alpha\beta} D_\beta((\chi-1) u).
$$
For $\tau \in [0,\infty)$, set
$$
\widetilde{B}_r^+(\tau,0) = B_r(\tau,0) \cap \{x_1 > \tau\},
$$
where $0 \in \bR^{d-1}$.
Since $\rho\in (0,1/16)$, we have that
$$
\Omega_{R/2} \subset \Omega_{3R/4}(\rho R,0)\quad
\text{and}
\quad
\widetilde{B}^+_{3R/4}(\rho R,0) \subset B_{R}\cap\{x_1>\rho R\}.
$$
By Lemma \ref{lem0225_1}, there is a unique solution
$$
(\hat w,p_1)\in W^{1}_2(\widetilde{B}^+_{3R/4}(\rho R,0))^d \times L_2(\widetilde{B}^+_{3R/4}(\rho R,0))
$$
satisfying $(p_1)_{\widetilde{B}^+_{3R/4}(\rho R,0)}=0$ and
\begin{equation}
							\label{eq0714_03}
\begin{cases}
\cL_0 \hat w+\nabla p_1 = D_\alpha (\tilde f_\alpha+h_\alpha)
\quad
&\text{in}\,\,\widetilde{B}^+_{3R/4}(\rho R,0),
\\
\operatorname{div} \hat w = \chi g+\chi' u_1-\big(\chi g+\chi' u_1\big)_{\widetilde{B}^+_{3R/4}(\rho R,0)}
\quad
&\text{in}\,\,\widetilde{B}^+_{3R/4}(\rho R,0),
\\
\hat w = 0
\quad
&\text{on} \,\, \partial \widetilde{B}^+_{3R/4}(\rho R,0).
\end{cases}
\end{equation}
Moreover, it holds that
\begin{align}
                                \label{eq3.01}
&\|D\hat w\|_{L_2} \le N\big(\|\tilde f_\alpha\|_{L_2}
+\|h_\alpha\|_{L_2}
+\|\chi g+\chi' u\|_{L_2}\big)\nonumber\\
&\le N\big(\|f_\alpha\|_{L_2}
+\|(\bar A^{\alpha\beta}- A^{\alpha\beta})D_\beta  u\|_{L_2}\nonumber\\
&\quad+\|D((\chi-1)u)\|_{L_2}
+\|g\|_{L_2}+\|\chi' u\|_{L_2}\big),
\end{align}
where $\|\cdot\|_{L_2}=\|\cdot\|_{L_2(\widetilde{B}^+_{3R/4}(\rho R,0))}$ and $N = N(d,\delta)$.
Using the fact that $|A^{\alpha\beta}| \le \delta^{-1}$, together with \eqref{eq17_50} and H\"older's inequality, it follows that
\begin{equation}
							\label{eq0329_01}
\|(\bar A^{\alpha\beta}- A^{\alpha\beta})D_\beta  u\|_{L_2(\widetilde{B}^+_{3R/4}(\rho R,0))}\le N\gamma^{\frac 1 {2\nu}}R^{\frac d {2\nu}}
\|Du\|_{L_{2\mu}(\widetilde{B}^+_{3R/4}(\rho R,0))}.
\end{equation}
Since $\chi-1$ is supported on $\{x_1\le 2\rho R\}$, H\"older's inequality implies that
\begin{equation}
							\label{eq0328_06}
\|(\chi-1)Du\|_{L_2(\widetilde{B}^+_{3R/4}(\rho R,0))} \le N\rho^{\frac 1 {2\nu}}R^{\frac d {2\nu}}
\|Du\|_{L_{2\mu}(\Omega_{R})}.
\end{equation}
Using H\"older's inequality again, together with the fact that $\chi'$ is supported on $\{\rho R \le x_1 \le 2\rho R\}$, we have that
\begin{align}
							\label{eq0328_07}
\|\chi' u\|_{L_2(\widetilde{B}^+_{3R/4}(\rho R,0))}
&\le N \rho^{\frac{1}{2\nu}} R^{\frac{d}{2\nu}}
\|\chi' u\|_{L_{2\mu}(\widetilde{B}^+_{3R/4}(\rho R,0))}\nonumber\\
&\le N \rho^{\frac{1}{2\nu}} R^{\frac{d}{2\nu}} \|Du\|_{L_{2\mu}(\Omega_R)},
\end{align}
where the last inequality follows from Hardy's inequality, using the boundary condition $u = 0$ on $\partial\Omega$ and the observation that
$$
|\chi'| \le N(x_1 - \phi(x'))^{-1}
$$
for $(x_1,x') \in \Omega_R$, where $\Omega_R$ is given by $\{ x \in B_R: x_1 > \phi(x')\}$.
The inequalities \eqref{eq0329_01}, \eqref{eq0328_06}, and \eqref{eq0328_07}, together with \eqref{eq3.01}, imply that
\begin{equation}
            \label{eq21.52h}
(|D \hat w|^2)_{\widetilde{B}^+_{3R/4}(\rho R,0)}^{\frac 1 2}
\le N(\rho+\gamma)^{\frac 1 {2\nu}} (|Du|^{2\mu})_{\Omega_R}^{\frac 1 {2\mu}}+
N(f_\alpha^2+g^2)_{\Omega_{R}}^{\frac 1 2}.
\end{equation}
We extend $\hat w$ to be zero  in $\Omega_{3R/4}(\rho R,0)\cap \{x_1<\rho R\}$, so that $\hat w\in W^{1}_2(\Omega_{3R/4}(\rho R,0))$, and we let
$$
w=\hat w+(1-\chi)u.
$$
By the same reasoning as in \eqref{eq0328_06} and \eqref{eq0328_07}, we have that
$$
\|D\left((1-\chi) u\right)\|_{L_2(\Omega_{3R/4}(\rho R,0))} \le N \rho^{\frac{1}{2\nu}} R^{\frac{d}{2 \nu}} \|Du\|_{L_{2\mu}(\Omega_R)}.
$$
From this and \eqref{eq21.52h}, we deduce that
\begin{equation}
            \label{eq18.34h}
(|Dw|^2)_{\Omega_{3R/4}(\rho R,0)}^{\frac 1 2}
\le N(\gamma+\rho)^{\frac 1 {2\nu}} (|Du|^{2\mu})_{\Omega_{R}}^{\frac 1  {2\mu}}+
N(f_\alpha^2+g^2)_{\Omega_{R}}^{\frac 1 2}.
\end{equation}
Note that, because $\kappa \rho\le 1/4$, it holds that
$$
\Omega_r(x_0) \subset \Omega_{3R/4}(\rho R,0)\quad \text{and} \quad |\Omega_{3R/4}(\rho R,0)|/|\Omega_r(x_0)|
\le N(d) \kappa^{d}.
$$
Thus, from \eqref{eq18.34h} we also obtain that
\begin{equation}
                                \label{eq28_01}
(|Dw|^2)_{\Omega_r(x_0)}^{\frac 1 2}
\le N\kappa^{\frac d 2}(\gamma+\rho)^{\frac 1 {2\nu}} (|Du|^{2\mu})_{\Omega_{R}}^{\frac 1 {2\mu}}+N \kappa^{\frac d 2}
(f_\alpha^2+g^2)_{\Omega_{R}}^{\frac 1 2}.
\end{equation}

Next, we define $v= u- w$ $(=\chi u-\hat w)$ in $\Omega_{3R/4}(\rho R,0)$ and $p_2=p-p_1$ in $\widetilde{B}^+_{3R/4}(\rho R,0)$.
From \eqref{eq17.23b} and \eqref{eq0714_03}, it is easily seen that $(v,p_2)$ satisfies
\begin{equation*}
\begin{cases}
\cL_0 v+\nabla p_2 = 0
\quad
&\text{in}\,\,\widetilde{B}^+_{3R/4}(\rho R,0),
\\
\operatorname{div} v = \big(\chi g+\chi' u_1\big)_{\widetilde{B}^+_{3R/4}(\rho R,0)}
\quad
&\text{in}\,\,\widetilde{B}^+_{3R/4}(\rho R,0),
\\
v = 0
\quad
&\text{on} \,\, B_{3R/4}(\rho R,0)\cap \{x_1=\rho R\}.
\end{cases}
\end{equation*}
Denote
$$
\cD_1=\Omega_{r}(x_0)\cap \{x_1\le \rho R\},
\quad
\cD_2=\Omega_{r}(x_0)\cap \{x_1> \rho R\},
\quad\text{and}\,\,
\cD_3=\widetilde{B}^+_{R/4}(\rho R,0).
$$
We have that $\cD_2\subset \cD_3$ and $|\cD_1|\le N\kappa\rho|\Omega_{r}(x_0)|$, where the latter follows from the fact that $\cD_1 = \Omega_r(x_0) \cap \{ - \rho R \le x_1 \le \rho R\}$.
We set
$$
V_i = \sum_{j=1}^d\sum_{\beta = 1}^d \bar{A}_{ij}^{1\beta}D_\beta v_j,
\quad
i = 2, \ldots, d,
$$
where the coefficients $\bar{A}^{1\beta}(x_1)$ are taken from \eqref{eq0715_01}.
Note that $v=V_i=0$ in $\cD_1$.
Then, by applying Lemma \ref{lem0229_2} with a dilation,
we get that
\begin{align}
&\big(|V_i-(V_i)_{\Omega_r(x_0)}|\big)_{\Omega_r(x_0)}
+\big(|D_{x'}v-(D_{x'}v)_{\Omega_r(x_0)}|\big)_{\Omega_r(x_0)}\nonumber\\
&\le N r^{1/2}\big([V_i]_{C^{1/2}(\cD_2)}
+[D_{x'}v]_{C^{1/2}( \cD_2)}\big)
+N\kappa\rho \left(\|V_i\|_{L_\infty(\cD_2)} +\|D_{x'}v\|_{L_\infty(\cD_2)} \right)\nonumber\\
&\le N r^{1/2}\big([V_i]_{C^{1/2}(\cD_3)}
+[D_{x'}v]_{C^{1/2}( \cD_3)}\big)
+N\kappa\rho \left(\|V_i\|_{L_\infty(\cD_3)} +\|D_{x'}v\|_{L_\infty(\cD_3)} \right)\nonumber\\
                        \label{eq5.01}
&\le N(\kappa^{-1/2}+\kappa \rho)(|Dv|^2)_{\Omega_{R/2}(\rho R,0)}^{1/2}.
\end{align}

Now, we set $\cU=(D_{x'}u,\operatorname{div} u, U_2,\ldots,U_d)$, where
\begin{equation}
							\label{eq0328_05}
U_i = \sum_{j=1}^d\sum_{\beta = 1}^d \bar{A}_{ij}^{1\beta}D_\beta u_j,
\quad
i = 2, \ldots, d.
\end{equation}
Note that $\cU$ satisfies \eqref{eq10.54}.
From the triangle inequality, \eqref{eq5.01}, and the fact that $|\Omega_{2R}(x_0)| \le N \kappa^d |\Omega_r(x_0)|$ by the condition that $\kappa \rho \le 1/4$, we have that
\begin{align*}
&\big(|\cU-(\cU)_{\Omega_r(x_0)}|\big)_{\Omega_r(x_0)}
\le N\big(|D_{x'}v-(D_{x'}v)_{\Omega_r(x_0)}|\big)_{\Omega_r(x_0)}\\
&\quad + N \big(|V_i-(V_i)_{\Omega_r(x_0)}|\big)_{\Omega_r(x_0)}
+N(|g|)_{\Omega_r(x_0)}+ N \big(|Dw|\big)_{\Omega_r(x_0)}\\
&\le N(\kappa^{-1/2}+\kappa \rho)(|Dv|^2)_{\Omega_{R/2}(\rho R, 0)}^{1/2}+N(|g|^2)^{1/2}_{\Omega_r(x_0)}
+N \big(|Dw|^2\big)^{1/2}_{\Omega_r(x_0)}\\
&\le N(\kappa^{-1/2}+\kappa \rho)(|Du|^2)_{\Omega_{R/2}(\rho R, 0)}^{1/2}
+N(\kappa^{-1/2}+\kappa \rho)(|Dw|^2)_{\Omega_{R/2}(\rho R, 0)}^{1/2}\\
&\quad +N\kappa^{d/2}(|g|^2)^{1/2}_{\Omega_{2R}(x_0)}+N \big(|Dw|^2\big)^{1/2}_{\Omega_r(x_0)},
\end{align*}
where we bound the second and last terms on the right-hand side of the last inequality by
using \eqref{eq18.34h} and \eqref{eq28_01}.
This completes the proof of Lemma \ref{lem6.5}.
\end{proof}

Before we present the proof of Theorem \ref{thm0307_2}, we note that a Lipschitz domain $\Omega$ in $\bR^d$ satisfying Assumption \ref{assum0620} with $\operatorname{diam}\Omega \le K$ is a space of homogeneous type, which is endowed with the Euclidean distance and a doubling measure $\mu$ that is naturally inherited from the Lebesgue measure. Owing to a result by Christ \cite[Theorem 11]{MR1096400} (also see \cite{HK12}), there exists a filtration of partitions of $\Omega$ in the following sense. For each $n\in\bZ$, there exists a collection of disjoint open subsets $\bC_n:=\{Q_\alpha^n\,:\,\alpha\in I_n\}$ for some index set $I_n$, satisfying the following properties:
\begin{enumerate}
\item For any $n\in \bZ$, $\mu(\Omega\setminus \bigcup_{\alpha}Q_\alpha^n)=0$;
\item For each $n$ and $\alpha\in I_n$, there exists a unique $\beta\in I_{n-1}$ such that $Q_\alpha^n\subset Q_\beta^{n-1}$;
\item For each $n$ and $\alpha\in I_n$, $\text{diam}(Q_\alpha^n)\le N_0\delta_0^n$;
\item Each $Q_\alpha^n$ contains some ball $\Omega_{\varepsilon_0\delta_0^n}(z_\alpha^n)$;
\end{enumerate}
for some constants $\delta_0\in (0,1)$, $\varepsilon_0>0$, and $N_0$ depending only on $d$, $R_0$, and $K$.

For any $f\in L_{1,\text{loc}}(\Omega)$, recall the definition of a maximal function in \eqref{eq0328_01} with a bounded Lipschitz domain $\Omega$ in place of $\bR^d$ or $\bR^d_+$:
$$
\cM f(x)=\sup_{x_0\in \bar{\Omega},\Omega_r(x_0)\ni x}\dashint_{\Omega_r(x_0)}|f(y)|\,dy.
$$
\begin{lemma}
                        \label{lem7.2}
Let $q \in (1,\infty)$ and $\Omega$ satisfy Assumption \ref{assum0620} with $\operatorname{diam}\Omega \le K$. Then, for any $f\in L_q(\Omega)$, we have that
$$
\|\cM f\|_{L_q(\Omega)}\le N\|f\|_{L_q(\Omega)},
$$
where $N>0$ is a constant depending only on $d$, $q$, $R_0$, and $K$.
\end{lemma}
\begin{proof}
Since $\Omega$ is a space of homogeneous type, the lemma follows from the Hardy-Littlewood maximal function theorem for spaces of homogeneous type. See, for instance, \cite{MR0740173}.
Also see \cite[Theorem 2.2]{DK15}.
\end{proof}

\begin{lemma}
                                \label{lem7.3}
Let $q \in (1,\infty)$  and $\Omega$ satisfy Assumption \ref{assum0620} with $\operatorname{diam}\Omega \le K$. Suppose that
$$
F\in L_q(\Omega),\quad G\in L_q(\Omega), \quad H\in L_q(\Omega),\quad |F| \le H,
$$
and that for each $n \in \bZ$ and $Q \in \bC_n$,
there exists a measurable function $F^Q$ on $Q$
such that
\begin{equation}							 \label{eq10.53}
|F| \le F^Q \le H\quad \text{on}\,\, Q\quad\text{and}\,\,
\dashint_Q |F^Q(x) - \left(F^Q\right)_Q| \,dx
\le N_0 G(y)\quad \forall \,y\in Q
\end{equation}
for some constant $N_0>0$.
Then, we have that
\begin{equation}
                                \label{eq12.09}
\|F\|_{L_q(\Omega)}^p
\le NN_0\|G\|_{L_q(\Omega)}\|H\|_{L_q(\Omega)}^{p-1}
+N\|H\|_{L_1(\Omega)}^{p},
\end{equation}
where $N>0$ is a constant depending only on $d$, $q$, $R_0$, and $K$.
\end{lemma}
\begin{proof}
This lemma is a special case of Theorem 2.4 of \cite{DK15}, in which $A_q$ weights are considered.
When there is no weight as in the lemma, it is easily seen that $\beta$ in that theorem is equal to $1$.
\end{proof}

We are now ready to present the proof of Theorem \ref{thm0307_2}.

\begin{proof}[Proof of Theorem \ref{thm0307_2}]
We only derive the a priori estimate \eqref{eq0328_02}. The solvability then follows from \eqref{eq0328_02}, the Poincar\'e inequality, and the method of continuity.
Furthermore, we assume that $f \equiv 0$.
Otherwise, for $B_R \supseteq \Omega$, we find $w \in W_{q_1}^2(B_R)$ such that $\Delta w = f 1_{\Omega}$ in $B_R$ and $w|_{\partial B_R} = 0$.
Then, we consider
$$
\cL u + \nabla p = D_\alpha\left( D_\alpha w + f_\alpha\right),
$$
where from the Sobolev embedding theorem and the well-known $L_{q_1}$-estimate for the Laplace equation we have that
$$
\|D_\alpha w\|_{L_q(\Omega)} \le \|D_\alpha w\|_{L_q(B_R)} \le N \|w\|_{W_{q_1}^2(B_R)}
\le N\|f\|_{L_{q_1}(\Omega)}.
$$
We consider the two cases with $q >2$ and $q\in (1,2)$. The case with $q=2$ follows from Lemma \ref{lem0225_1}.

\noindent{\bf Case 1}: $q>2$.
We take $\mu\in (1,\infty)$, depending only on $q$, such that $2\mu<q$, and we let $\kappa\ge 32$ be a constant to be specified. By the properties (3) and (4) described above, for each $Q$ in the partitions
there exist $r \in (0,\infty)$ and $x_0 \in \bar\Omega$ such that
\begin{equation}
                                \label{eq11.23}
Q \subset \Omega_r(x_0)\quad\text{and}
\quad |\Omega_r(x_0)| \le N |Q|,
\end{equation}
where $N$ depends only on $d$, $R_0$, and $K$. See Remark 7.3 in \cite{DK15}.
In order to apply Lemma \ref{lem7.3}, we take $F=|Du|$, $H=N|Du|$, where $N = N(d,\delta,q) \ge 1$ from \eqref{eq10.54}, and
\begin{align*}
G(y) &= (\kappa^{-\frac 1 2} + \kappa \rho) \left[ \cM(|Du|^{2})(y)\right]^{\frac{1}{2}}
+ \kappa^{\frac{d}{2}} \left[ \cM(f_\alpha^{2})(y)\right]^{\frac{1}{2}}\\
&\quad +\kappa^{\frac{d}{2}}\left[ \cM(g^{2})(y)\right]^{\frac{1}{2}}+ \kappa^{\frac{d}{2}} (\gamma+\rho)^{\frac{1}{2\nu}} \left[ \cM (|Du|^{2\mu})(y) \right]^{\frac{1}{2\mu}}+R_1^{-d}\kappa^{d}\|Du\|_{L_1(\Omega)}.
\end{align*}
For $F^Q$, we consider two cases. When $\kappa r\le R_1$, we choose $F^Q=\cU$, where $\cU$ is from Lemma \ref{lem6.5}.
Thanks to \eqref{eq10.54}, \eqref{eq3.33}, and \eqref{eq11.23}, we have that \eqref{eq10.53} holds with $N_0$ depending only on $d$, $\delta$, $R_0$, $K$, and $q$. Otherwise, i.e., if $r>R_1/\kappa$ we take $F^Q=|Du|$.
Then, by \eqref{eq11.23} we have
$$
\dashint_Q |F^Q(x) - \left(F^Q\right)_Q| \,dx
\le N\dashint_{\Omega_r(x_0)} |Du| \,dx
\le N|\Omega_{R_1/\kappa}(x_0)|^{-1}\|Du\|_{L_1(\Omega)},
$$
where $N = N(d, R_0, K)$.
Since $|\Omega_{R_1/\kappa}(x_0)|^{-1}\le NR_1^{-d}\kappa^{d}$, we still get that \eqref{eq10.53} holds with $N_0$ depending only on $d$, $R_0$, and $K$. Therefore, the conditions in Lemma \ref{lem7.3} are satisfied. From \eqref{eq12.09}, we obtain that
\begin{align*}
&\|Du\|_{L_q(\Omega)}\le N\|G\|_{L_q(\Omega)}+N\|Du\|_{L_1(\Omega)}\\
&\le NR_1^{-d}\kappa^{d}\|Du\|_{L_1(\Omega)}+N(\kappa^{-\frac 1 2} + \kappa \rho)
\|\cM(|Du|^{2})^{\frac{1}{2}}\|_{L_q(\Omega)}
+N\kappa^{\frac{d}{2}}\|\cM(f_\alpha^{2})^{\frac{1}{2}}\|_{L_q(\Omega)}\\
&\quad +N\kappa^{\frac{d}{2}}\|\cM(g^{2})^{\frac{1}{2}}\|_{L_q(\Omega)}
+N\kappa^{\frac{d}{2}} (\gamma+\rho)^{\frac{1}{2\nu}}
\|\cM(|Du|^{2\mu})^{\frac{1}{2\mu}}\|_{L_q(\Omega)}\\
&= NR_1^{-d}\kappa^{d}\|Du\|_{L_1(\Omega)}+N(\kappa^{-\frac 1 2} + \kappa \rho)
\|\cM(|Du|^{2})\|_{L_{q/2}(\Omega)}^{\frac{1}{2}}
+N\kappa^{\frac{d}{2}}\|\cM(f_\alpha^{2})\|^{\frac{1}{2}}_{L_{q/2}(\Omega)}\\
&\quad +N\kappa^{\frac{d}{2}}\|\cM(g^{2})\|^{\frac{1}{2}}_{L_{q/2}(\Omega)}
+N\kappa^{\frac{d}{2}} (\gamma+\rho)^{\frac{1}{2\nu}}
\|\cM(|Du|^{2\mu})\|^{\frac{1}{2\mu}}_{L_{q/{2\mu}}(\Omega)},
\end{align*}
where $N = N(d,\delta,R_0,K,q)$.
By Lemma \ref{lem7.2}, the right-hand side above is bounded by
\begin{align*}
&NR_1^{-d}\kappa^{d}\|Du\|_{L_1(\Omega)}+N(\kappa^{-\frac 1 2} + \kappa \rho)
\||Du|^{2}\|_{L_{q/2}(\Omega)}^{\frac{1}{2}}
+N\kappa^{\frac{d}{2}}\|f_\alpha^{2}\|^{\frac{1}{2}}_{L_{q/2}(\Omega)}\\
&\quad +N\kappa^{\frac{d}{2}}\|g^{2}\|^{\frac{1}{2}}_{L_{q/2}(\Omega)}
+N\kappa^{\frac{d}{2}} (\gamma+\rho)^{\frac{1}{2\nu}}
\||Du|^{2\mu}\|^{\frac{1}{2\mu}}_{L_{q/{2\mu}}(\Omega)}\\
&=NR_1^{-d}\kappa^{d}\|Du\|_{L_1(\Omega)}+N(\kappa^{-\frac 1 2} + \kappa \rho)
\|Du\|_{L_q(\Omega)}
+N\kappa^{\frac{d}{2}}\|f_\alpha\|_{L_q(\Omega)}\\
&\quad +N\kappa^{\frac{d}{2}}\|g\|_{L_q(\Omega)}
+N\kappa^{\frac{d}{2}} (\gamma+\rho)^{\frac{1}{2\nu}}
\|Du\|_{L_q(\Omega)}.
\end{align*}
Upon taking sufficiently large $\kappa$, then sufficiently small $\gamma$ and $\rho$, depending on $d$, $\delta$, $R_0$, $K$, and $q$ (but independent of $R_1$), so that
$$
N(\kappa^{-\frac 1 2} + \kappa \rho)+N\kappa^{\frac{d}{2}} (\gamma+\rho)^{\frac{1}{2\nu}}\le 1/2,
$$
we get that
\begin{equation}
							\label{eq0328_04}
\|Du\|_{L_q(\Omega)}\le NR_1^{-d}\kappa^{d}\|Du\|_{L_1(\Omega)}+N\|f_\alpha\|_{L_q(\Omega)}+N\|g\|_{L_q(\Omega)}.
\end{equation}
Since $\Omega$ is bounded, we have that $f_\alpha,g\in L_2(\Omega)$. Thus, by Lemma \ref{lem0225_1} (also see Remark \ref{rem0229_1}) and H\"older's inequality,
$$
\|Du\|_{L_2(\Omega)} + \|p\|_{L_2(\Omega)} \le N \left(\|f_\alpha\|_{L_q(\Omega)} + \|g\|_{L_q(\Omega)} \right),
$$
where $N=N(d,\delta,R_0, K, \rho,q)$. Combining this with \eqref{eq0328_04} yields that
\begin{equation}
                    \label{eq10.09}
\|Du\|_{L_q(\Omega)} \le N \left( \|f_\alpha\|_{L_q(\Omega)} + \|g\|_{L_q(\Omega)} \right),
\end{equation}
where $N = N(d,\delta,R_0,R_1,K,q)$.
Next, we estimate $p$. For any $\eta\in L_{q'}(\Omega)$ with $q'=q/(q-1)$, it follows from the solvability of the divergence equation in Lipschitz domains (cf. \cite{MR2101215}) that there exists $\psi\in \mathring{W}_{q'}^1(\Omega)^d$ such that
\begin{equation}
                            \label{eq10.10}
\operatorname{div} \psi = \eta-(\eta)_\Omega
\quad
\text{in}
\,\,
\Omega
\quad
\text{and}
\quad
\|D\psi\|_{L_{q'}(\Omega)} \le N \|\eta\|_{L_{q'}(\Omega)},
\end{equation}
where $N>0$ is a constant depending only on $d$, $R_0$, $K$, and $q'$.
We test the first equation of \eqref{eq0307_04} by $\psi$, and using the fact that $(p)_\Omega = 0$ we obtain
$$
\int_\Omega p\eta\,dx=\int_\Omega p\big(\eta-(\eta)_\Omega\big)\,dx
=\int_\Omega D_\alpha \psi \cdot (f_\alpha-A^{\alpha\beta} D_\beta u) \,dx,
$$
which combined with \eqref{eq10.09} and \eqref{eq10.10} yields
$$
\left|\int_\Omega p\eta\,dx\right|\le N\left( \|f_\alpha\|_{L_q(\Omega)} + \|g\|_{L_q(\Omega)} \right)\|\eta\|_{L_{q'}(\Omega)}.
$$
Since $\eta\in L_{q'}(\Omega)$ is arbitrary, we can infer that
$$
\|p\|_{L_q(\Omega)} \le N \left( \|f_\alpha\|_{L_q(\Omega)} + \|g\|_{L_q(\Omega)} \right),
$$
which together with \eqref{eq10.09} implies that \eqref{eq0328_02} holds.

\noindent{\bf Case 2}: $q\in (1,2)$. We employ a duality argument. Let $q'=q/(q-1)\in (2,\infty)$ and $(\gamma, \rho) = (\gamma, \rho)(d,\delta,R_0,K,q')$ from Case 1. Then, for any $\eta=(\eta_{\alpha})$, where $\eta_\alpha\in L_{q'}(\Omega)^d$ for $\alpha=1,\ldots,d$, there exists a unique solution $(v,\pi) \in W_{q'}^{1}(\Omega)^d \times L_{q'}(\Omega)$ with $(\pi)_\Omega = 0$ satisfying
$$
\begin{cases}
D_\beta(A^{\alpha\beta}_{\operatorname{tr}} D_\alpha v) + \nabla \pi = D_\alpha \eta_\alpha
\quad
&\text{in}\,\,\Omega,
\\
\operatorname{div} v = 0
\quad
&\text{in}\,\,\Omega,
\\
v= 0\quad
&\text{on}\,\,\partial\Omega,
\end{cases}
$$
where $A^{\alpha\beta}_{\operatorname{tr}}$ is the transpose of the matrix $A^{\alpha\beta}$ for each $\alpha,\beta=1,\ldots,d$.
Moreover, we have that
\begin{equation}
                                \label{eq10.55}
\|Dv\|_{L_{q'}(\Omega)} + \|\pi\|_{L_{q'}(\Omega)} \le N \|\eta_\alpha\|_{L_{q'}(\Omega)},
\end{equation}
where $N = N(d,\delta,R_0,R_1,K,q')$. Then, we test the equation of $(v,\pi)$ by $u$, to obtain
$$
\int_\Omega  \eta_\alpha  \cdot D_\alpha u\,dx=\int_\Omega \big( D_\beta u \cdot A^{\alpha\beta}_{\operatorname{tr}}  D_\alpha v +\pi\, \operatorname{div} u\big)\,dx=
\int_\Omega \big(f_\alpha \cdot D_\alpha v+\pi g\big)\,dx.
$$
From this and \eqref{eq10.55}, we get that
$$
\left|\int_\Omega \eta_\alpha \cdot D_\alpha u\,dx\right|\le N\|\eta_\alpha\|_{L_{q'}(\Omega)}\big(\|f_\alpha\|_{L_q(\Omega)}
+\|g\|_{L_q(\Omega)}\big).
$$
Since $\eta\in \big(L_{q'}(\Omega)\big)^{d^2}$ is arbitrary, we obtain that
\eqref{eq10.09} holds. The estimate of $p$ is the same as in Case 1. Thus, the theorem is proved.
\end{proof}

%

\bibliographystyle{plain}

\def\cprime{$'$}

\end{document}